\documentclass[reqno]{amsart} 
\usepackage[utf8]{inputenc} 

\usepackage{calrsfs}

\usepackage{amssymb,amsthm,amsfonts} 
\usepackage{mathtools} 
\usepackage{fixltx2e} 
\usepackage{accents} 
\usepackage{tikz} 
\usepackage{graphicx} 
\usepackage{psfrag} 
\usepackage{cancel} 
\usepackage{hyperref} 
\usepackage{comment} 
\usepackage{ulem}

\newtheorem{theorem}{Theorem}[section] 
\newtheorem{lemma}{Lemma}[section] 
\newtheorem{proposition}{Proposition}[section] 
\newtheorem{corollary}{Corollary}[section] 
\newtheorem{remark}{Remark}[section]

\newtheorem{claim}{Claim}
\numberwithin{equation}{section}

\newlength{\dhatheight}

\newcommand{\caret}{\mathbin{\scriptscriptstyle\wedge}}

\newcommand{\pq}[2]{\ensuremath{#1,#2}}

\makeatletter
\def\revddots{\mathinner{\mkern1mu\raise\p@
\vbox{\kern7\p@\hbox{.}}\mkern2mu
\raise4\p@\hbox{.}\mkern2mu\raise7\p@\hbox{.}\mkern1mu}}
\makeatother

\newcommand{\kay}{\ensuremath{k}}

\newcommand{\F}{\ensuremath{\mathsf{F}}}

\newcommand{\A}{\ensuremath{\mathsf{A}}}

\renewcommand{\mod}{\ensuremath{\, \, \mathrm{mod} \,}}

\newcommand{\card}{\ensuremath{\#}}
\newcommand{\id}{\ensuremath{\mathrm{id}}}

\DeclareMathOperator{\discr}{discr}
\DeclareMathOperator{\trace}{tr}
\DeclareMathOperator{\spec}{spec}
\DeclareMathOperator{\specrad}{rad}

\title[]{Quadratic Irrationals, Closed Geodesics on the Modular Surface and Dynamical Zeta Functions}


\thanks{
This work has been partially supported by CAPES Special Visiting Researcher grant CSF-PVE-S - 88887.117899/2016-00.
}

\author{Peter Hazard} 
\address{Peter Hazard, Instituto de Matem\'{a}tica e Estat{\'i}stica, USP, S\~{a}o Paulo, SP, Brazil}
\email[]{pete@ime.usp.br}

\date{\today}

\begin{document}

\begin{comment}


\noindent
{\it Comments.}
The following are some changes that have been made or need making.

\vspace{5pt}

{\it [2017-06-25]}
\begin{enumerate}
\item[(0)]
Kept several sections from part I to maintain consistent notation etc.
\end{enumerate}

\newpage

\pagestyle{plain}
\setcounter{page}{1}
\end{comment}

\begin{abstract}
We show that generating functions associated to the sequence of convergents of a quadratic irrational
are related in a natural way to the dynamical zeta function of a hyperbolic automorphism of the $2$-torus.
As a corollary, this shows that the L\'{e}vy constant of a quadratic irrational appears naturally as the topological entropy of such maps. 
\end{abstract}

\maketitle

\section{Introduction}
\subsection{Background.}
Given an irrational real number $\theta$, 
for each non-negative integer $n$ let the rational real number $p_n/q_n$
denote the $n$th convergent of $\theta$ 
({\it i.e.}, the $n$th best rational approximant). 
The following limit, when it exists, is called 
the {\it L{\'e}vy constant} of $\theta$:
\begin{equation}\label{def:Levy_const}
\beta(\theta)=\lim_{n\to\infty}\frac{1}{n}\log q_n
\end{equation}
This is named after Paul L{\'e}vy, 
who proved that~\cite{Levy1929} 
\begin{equation}\label{thm:PaulLevy}
\beta(\theta)=\frac{\pi^2}{12\log 2}
\qquad
\mbox{for Lebesgue almost every} \quad \theta\in[0,1]
\end{equation}
(For more information see~\cite{LevyBook,KhinchinBook,Lehmer1939} and the references therein.)
Note that this can be shown by using the ergodicity
of the Gauss transformation with respect to the 
invariant Gauss measure together with the Birkhoff Ergodic Theorem.
 
For quadratic irrationals, it was shown by Jager and Liardet~\cite{JagerLiardet1988} that the above limit~\eqref{def:Levy_const} exists. 
However, in this case, the L\'{e}vy constant is not generally given by~\eqref{thm:PaulLevy}.
In fact, for a quadratic irrational $\theta$, the L\'evy constant is given by 
\begin{equation}\label{eq:Levy-radius_of_conv}
\beta(\theta)=\frac{1}{\ell}\log \mathrm{rad}(N_\theta)
\end{equation}
where 
$\ell$ denotes the length of the (eventual) period of the continued fraction expansion of $\theta$,
$N_\theta$ denotes an element of $\mathrm{PSL}(2,\mathbb{C})$ (in fact of $\mathrm{PSL}(2,\mathbb{Z})$) 
associated to the continued fraction expansion of $\theta$, 
and $\mathrm{rad}(A)$ denotes the spectral radius of either of the linear transformations corresponding to $A\in\mathrm{PSL}(2,\mathbb{C})$.
(For the readers' convenience, we recall the ergodic-theoretic proof of the L\'evy theorem and the Jager-Liardet theorem in Appendix~\ref{sect:Gauss_Map}.)

In~\cite{BelovaHazard2017a}, 
a new proof of the result of Jager and Liardet was given as a corollary of the following.
\begin{theorem}\label{thm:gen_fn_rat}
Let $\theta\in[0,1]\setminus \mathbb{Q}$ 
be a quadratic irrational with continued fraction expansion
\begin{equation}\label{eq:theta_ctd_frac_expansion}
\theta=\left[a_1,a_2,\ldots,a_\kay,\overline{a_{\kay+1},a_{\kay+2},\ldots,a_{\kay+\ell}}\right]
\end{equation}
Let 
\begin{align}
N_0=
\left[\begin{array}{cc}0&1\\1&a_1\end{array}\right]
\cdots
\left[\begin{array}{cc}0&1\\1&a_\kay\end{array}\right] 
\qquad
N_1=
\left[\begin{array}{cc}0&1\\1&a_{\kay+1}\end{array}\right]
\cdots
\left[\begin{array}{cc}0&1\\1&a_{\kay+\ell}\end{array}\right]
\end{align}
For each non-negative integer $n$, 
let $p_n/q_n$ denote the $n$th convergent of $\theta$.
Then the associated generating functions
\begin{equation}
F_p(z)=\sum_{n\geq 0} p_n z^n \qquad
F_q(z)=\sum_{n\geq 0} q_n z^n
\end{equation}
are both rational functions in the variable $z$, with integer coefficients.
In fact, 
\begin{equation}\label{eq:gen_fn_rat}
\left[\begin{array}{c}F_p(z)\\ F_q(z)\end{array}\right]
=
\sum_{0\leq n<\kay}z^{n}\left[\begin{array}{c}p_n\\ q_n\end{array}\right]
+\Bigl(\mathrm{id}-z^{\ell}N_0N_1N_0^{-1}\Bigr)^{-1} 
\sum_{\kay\leq n<\kay+\ell}z^n
\left[\begin{array}{c}p_n\\ q_n\end{array}\right]
\end{equation}
\end{theorem}
(Note that restricting to the case of $\theta\in [0,1]\setminus\mathbb{Q}$ is merely a simplification.
This result, as well as those given in this article, hold more generally for quadratic irrationals in $\mathbb{R}\setminus\mathbb{Q}$.)
In this paper we prove a generalisation of the above result to more general generating functions 
whose coefficients are monomials in the $p_n$ and $q_n$ (Theorem~\ref{thm:gen_fn_rat->higher-powers}). 
We also show that to each quadratic irrational $\theta$, 
there exists a hyperbolic toral automorphism $f_\theta$ of $\mathbb{T}^2$
such that the generating functions defined above are related in a 
natural way to the dynamical zeta function of $f_\theta$.
Below we prove the following result.
\begin{theorem}\label{thm:main_thm}
Let $\theta\in [0,1]\setminus\mathbb{Q}$ be a quadratic irrational.
Then there exists 
\begin{itemize}
\item
a hyperbolic toral automorphism $f$ of $\mathbb{T}^2$, depending on $\theta$ only, 
\item
a rational function $R$ on $\hat{\mathbb{C}}$, independent of $\theta$,
\item
non-degenerate matrices $U(z)$, $V(z)$, $X(z)$ and $Y(z)$ with entries in $\mathbb{Z}(z)$ 
(the ring of rational functions in the variable $z$ with coefficients in $\mathbb{Z}$), 
\end{itemize}
such that the following equality holds
\begin{align}\label{eq:dlogzeta_vs_F-G}
z^\ell\left(\log\zeta_f\right)'(z^\ell)
+R(z^\ell)
=
\trace \left[U(z)^{-1}(X(z)-Y(z)) V(z)^{-1}\right]
\end{align}
where $\zeta_f$ denotes the dynamical zeta function of $f$.
\end{theorem}  
\noindent
Applying classical arguments already gives 
$(\log\zeta_f)'(z)=-R(z)+\trace (\id-z M_f)^{-1}$, 
where $M_f$ denotes the matrix corresponding to $f$. 
Simplifying the expression for the operator on the right-hand side of~\eqref{eq:dlogzeta_vs_F-G} recovers this result.
However, the reason for considering the equality in Theorem~\ref{thm:main_thm} 
is to relate the set of poles of $F_p$ and $F_q$
to the set of poles of $\zeta_f$ or
equivalently, the spectrum of the transfer operator for $f$. 

As one application of Theorem~\ref{thm:main_thm}, 
recall that the radius of convergence 
$\rho_{\zeta_f}$ 
of the dynamical zeta function $\zeta_f$ is related to the 
topological entropy of $f$ via
\begin{equation}\label{eq:entropy-radius_of_conv}
\frac{1}{\rho_{\zeta_f}}=\exp(h_\mathrm{top}(f))
\end{equation}
By equations~\eqref{eq:Levy-radius_of_conv} and~\eqref{eq:entropy-radius_of_conv}
we get the following as a Corollary.
\begin{theorem}\label{thm:Levy_vs_entropy}
Let 
$\theta\in\mathbb{R}\setminus\mathbb{Q}$ be a quadratic irrational.
Let 
$M_\theta$ 
denote the matrix corresponding to the prime hyperbolic element of 
$\mathrm{PSL}(2,\mathbb{Z})$ 
associated to $\theta$.
Let 
$f_\theta=f_{M_\theta}$ 
denote the induced hyperbolic toral automorphism.
Then
\begin{equation}
h_{\mathrm{top}}(f_\theta)=\ell \beta(\theta)
\end{equation} 
\end{theorem}
\noindent
What is meant by the 
{\it corresponding prime hyperbolic element} 
and also by the 
{\it induced hyperbolic toral automorphism} will be explained in the text.
Note that we also give a more direct proof of this result in Section~\ref{subsect:entropy+toral_automorphism}.
%
%
%
\subsection{Notation and Terminology.}
Let $\mathbb{N}$ and $\mathbb{N}_0$ denote the set of positive and non-negative integers respectively.
Let $\mathbb{Z}$ denote the set of integers and for each $\ell\in\mathbb{N}$, let $\mathbb{Z}_\ell=\mathbb{Z}/\ell\mathbb{Z}$.
Denote the rational, real and complex number fields by $\mathbb{Q}$, $\mathbb{R}$ and $\mathbb{C}$ respectively.
A {\it algebraic number} $\theta$ is a complex number which is the root of a non-zero polynomial with coefficients in 
$\mathbb{Q}$ (or equivalently, after clearing denominators, $\mathbb{Z}$).
An algebraic number is an {\it algebraic integer} if it is a root of some monic polynomial with coefficients in $\mathbb{Z}$.
Given an algebraic number $\theta$, the {\it minimal polynomial over $\mathbb{Q}$} 
is the unique monic polynomial with coefficients in $\mathbb{Q}$ of minimal degree, for which $\theta$ is a root.  
The {\it degree} of $\theta$ is the degree of its minimal polynomial.
Algebraic numbers with the same minimal polynomial are said to be {\it Galois conjugate}.
Given an arbitrary polynomial $\eta$, over either $\mathbb{R}$ or $\mathbb{C}$, we denote the discriminant by $\discr(\eta)$.

Given $\theta\in\mathbb{R}$, 
let $\lfloor\theta\rfloor$ denote the {\it integer part} of $\theta$, 
{\it i.e.\/}, greatest integer less than or equal to $\theta$, 
and let $\{\theta\}=\theta-\lfloor\theta\rfloor$ denote the {\it fractional part} of $\theta$.

Given an arbitrary set $S$ we denote its cardinality by $\card S$.
For a self-map $F$ of an arbitrary set $S$, we denote the set of fixed points by $\mathrm{Fix}(F)$.


Given a finite-dimensional vector space $V$, over $\mathbb{R}$ or $\mathbb{C}$, 
we denote the space of linear operators on $V$ by $L(V,V)$. 
Let $\mathrm{id}_V$, or just $\mathrm{id}$, denote the identity matrix and
given a linear map $A\in L(V,V)$, 
denote the transpose by $A^\top$.
We denote the trace by $\trace(A)$, the determinant by $\det(A)$.
Denote the spectrum of $A$ by $\spec(A)$, and 
the spectral radius by $\specrad(A)$.

\subsection*{Acknowledgements.}
The author would like to thank the Mathematics Institute at 
Uppsala University for their hospitality, and IME-USP for their continuing support.

\section{Preliminaries.}\label{sect:prelim}
%
%
%
%
%
\subsection{Continued fractions.}\label{subsect:ctd_frac_1}
Here we recall some basic properties of simple continued fraction expansions.
Our notation, and several of the results, are as in~\cite{BelovaHazard2017a}.
However, we include this again here for the readers convenience. 
For more details concerning continued fractions we recommend~\cite{KhinchinBook,HardyWrightBook}. 

%
\subsubsection{Best rational approximants.}\label{subsubsect:best_approx}
Let $\theta\in[0,1]\setminus\mathbb{Q}$.
The statements given below, with suitable modifications, also hold for general irrational points outside the unit interval.
However, to simplify the exposition we restrict ourselves to the case $\theta\in [0,1]\setminus\mathbb{Q}$.

The simple continued fraction expansion of $\theta\in[0,1]\setminus\mathbb{Q}$ is given by 
\begin{equation}\label{eq:theta-simple_ctd_frac_exp}
\theta
=\left[a_1,a_2,\ldots\right]
=\frac{1}{a_1+}\frac{1}{a_2+}\cdots\frac{1}{a_n+}\cdots \ ,
\end{equation}
where $a_1,a_2,\ldots$ are positive integers called the {\it partial quotients} of the simple continued fraction expansion.
Define the {\it $n$th convergent} of $\theta$ to be
\begin{equation}\label{eq:best_rat_approx}
[a_1,a_2,\ldots,a_n]
=\frac{1}{a_1+}\frac{1}{a_2+}\cdots\frac{1}{a_{n-1}+}\frac{1}{a_n} \ .
\end{equation}
We denote this rational number by
$p_n/q_n$, 
where $p_n$ and $q_n$ are positive integers having no common factors.
\begin{comment}
The following property is satisfied for all $n\in\mathbb{N}$
\begin{equation}
\left|\theta-\frac{p_n}{q_n}\right|\leq \inf_{\frac{p}{q}\in\mathbb{Q}: q\leq q_n}\left|\theta-\frac{p}{q}\right|
\end{equation}
For this reason $p_n/q_n$ is also called the {\it $n$th best rational approximant} of $\theta$.
\end{comment}
We will also call them the $n$th best rational approximants for obvious reasons (see~\cite{HardyWrightBook}).
Equation~\eqref{eq:best_rat_approx} may be expressed in matrix form as
\begin{align}
\left[\begin{array}{c}p_n\\ q_n\end{array}\right]
&=
\left[\begin{array}{cc}0 & 1\\ 1 & a_1\end{array}\right]
\left[\begin{array}{cc}0 & 1\\ 1 & a_2\end{array}\right]
\cdots
\left[\begin{array}{cc}0 & 1\\ 1 & a_{n-1}\end{array}\right]
\left[\begin{array}{c}1\\ a_{n}\end{array}\right]\label{eq:pq_n}
\end{align}
and similarly
\begin{align}
\left[\begin{array}{c}p_{n-1}\\ q_{n-1}\end{array}\right]
&=
\left[\begin{array}{cc}0 & 1\\ 1 & a_1\end{array}\right]
\left[\begin{array}{cc}0 & 1\\ 1 & a_2\end{array}\right]
\cdots
\left[\begin{array}{cc}0 & 1\\ 1 & a_{n-1}\end{array}\right]
\left[\begin{array}{c}0\\ 1\end{array}\right]\label{eq:pq_n-1}
\end{align}
Combining equations~\eqref{eq:pq_n} and~\eqref{eq:pq_n-1} therefore gives
\begin{equation}\label{eq:recursion_rel_simple}
\left[\begin{array}{cc}p_{n-1} & p_{n}\\ q_{n-1} & q_{n}\end{array}\right]
=
\left[\begin{array}{cc}0&1\\ 1&a_1\end{array}\right]
\left[\begin{array}{cc}0&1\\ 1&a_2\end{array}\right]
\cdots
\left[\begin{array}{cc}0 & 1\\ 1 & a_{n-1}\end{array}\right]
\left[\begin{array}{cc}0 & 1\\ 1 & a_{n}\end{array}\right]
\end{equation}
Hence inductively we get the following recurrence relations:
\begin{equation}
\left[\begin{array}{cc}p_{n-1} & p_{n}\\ q_{n-1} & q_{n}\end{array}\right]
=
\left[\begin{array}{cc}p_{n-2} & p_{n-1}\\ q_{n-2} & q_{n-1}\end{array}\right]
\left[\begin{array}{cc}0 & 1\\ 1 & a_{n}\end{array}\right],
\quad
\left[\begin{array}{cc}p_{0} & p_{1}\\ q_{0} & q_{1}\end{array}\right]
=
\left[\begin{array}{cc}0 & 1\\ 1 & a_1\end{array}\right]
\end{equation}
More generally, for any non-negative integer $m\leq n-2$,
\begin{equation}
\label{eq:recursion_rel_general}
\left[\begin{array}{cc}p_{n-1} & p_{n}\\ q_{n-1} & q_{n}\end{array}\right]
=
\left[\begin{array}{cc}p_{n-m-2} & p_{n-m-1}\\ q_{n-m-2}&q_{n-m-1}\end{array}\right]
\left[\begin{array}{cc}0 & 1\\ 1 & a_{n-m}\end{array}\right]
\cdots
\left[\begin{array}{cc}0 & 1\\ 1 & a_{n}\end{array}\right]
\end{equation}
For suitable integers $n$ and $m$ define
\begin{equation}
\label{eq:AB_definition}
\left[\begin{array}{cc}B^{(m)}_{n-m-1} & B^{(m+1)}_{n-m-1}\\ A^{(m)}_{n-m-1} & A^{(m+1)}_{n-m-1}\end{array}\right]
=
\left[\begin{array}{cc}0 & 1\\ 1 & a_{n-m}\end{array}\right]
\cdots
\left[\begin{array}{cc}0 & 1\\ 1 & a_{n}\end{array}\right]
\end{equation}
Then~\eqref{eq:recursion_rel_general} becomes
\begin{equation}\label{eq:rel_pq_AB}
\left[\begin{array}{cc}p_{n-1}& p_{n}\\ q_{n-1}&q_{n}\end{array}\right]
=
\left[\begin{array}{cc}p_{n-m-2} & p_{n-m-1}\\ q_{n-m-2} & q_{n-m-1}\end{array}\right]
\left[\begin{array}{cc}B^{(m)}_{n-m-1} & B^{(m+1)}_{n-m-1}\\ A^{(m)}_{n-m-1} & A^{(m+1)}_{n-m-1}\end{array}\right]
\end{equation}
Observe that the $A^{(m)}_{n-m}$ and $B^{(m)}_{n-m}$ are well-defined since we have the relation
\begin{equation}\label{eq:A_relation_1}
\left[\begin{array}{ll}
B^{(m)}_{n-m-1} & B^{(m+1)}_{n-m-1}\\
A^{(m)}_{n-m-1} & A^{(m+1)}_{n-m-1}
\end{array}\right]
=
\left[\begin{array}{cc}
0&1\\1&a_{n-m}
\end{array}\right]
\left[\begin{array}{ll}
B^{(m-1)}_{n-m} & B^{(m)}_{n-m}\\
A^{(m-1)}_{n-m} & A^{(m)}_{n-m}
\end{array}\right]
\end{equation}
Observe that the above is also implied by the corresponding dual relation
\begin{equation}\label{eq:A_relation_2}
\left[\begin{array}{ll}
B^{(m)}_{n-m-1} & B^{(m+1)}_{n-m-1}\\
A^{(m)}_{n-m-1} & A^{(m+1)}_{n-m-1}
\end{array}\right]
=
\left[\begin{array}{ll}
B^{(m-1)}_{n-m-1} & B^{(m)}_{n-m-1}\\
A^{(m-1)}_{n-m-1} & A^{(m)}_{n-m-1}
\end{array}\right]
\left[\begin{array}{cc}
0&1\\1&a_{n}
\end{array}\right]
\end{equation}
The reason for choosing this notation is the following.
The sequences $p_n$ and $q_n$ satisfy 
the recursion relation~\eqref{eq:recursion_rel_simple} which 
may be stated in the form
\begin{align}
p_{n}&=a_{n}p_{n-1}+p_{n-2};& p_0&=0& p_1&=1\label{eq:recurrence_rel_p}\\
q_{n}&=a_{n}q_{n-1}+q_{n-2};& q_0&=1& q_1&=a_1\label{eq:recurrence_rel_q}
\end{align}
Applying the recurrence relations inductively we more generally get the following expressions
\begin{equation}\label{eq:AB_expansion}
\left.\begin{array}{lcllcl}
p_{n}&=& A^{(1)}_{n-1}p_{n-1}+B^{(1)}_{n-1}p_{n-2} \qquad &q_{n}&=& A^{(1)}_{n-1}q_{n-1}+B^{(1)}_{n-1}q_{n-2}\\
     &=& A^{(2)}_{n-2}p_{n-2}+B^{(2)}_{n-2}p_{n-3} \qquad &     &=& A^{(2)}_{n-2}q_{n-2}+B^{(2)}_{n-2}q_{n-3}\\ 
     &\vdots&                                             &     &\vdots& \\
     &=& A^{(n-1)}_{1}p_{1}+B^{(n-1)}_{1}p_{0}     \qquad &     &=& A^{(n-1)}_{1}q_{1}+B^{(n-1)}_{1}q_{0}     
\end{array}\right.
\end{equation}
where 
$A^{(m)}_{n-m}$ and $B^{(m)}_{n-m}$ 
are non-negative integers satisfying the recurrence relations
\begin{equation}\label{eq:AB_recurrence_rel}
\begin{gathered}
B^{(m)}_{n-m}
=A^{(m-1)}_{n-m+1} 
\qquad
A^{(m)}_{n-m}
=B^{(m-1)}_{n-m+1}+a_{n-m+1}A^{(m-1)}_{n-m+1}\\
B^{(0)}_{n}=0 
\qquad 
B^{(1)}_{n-1}=1=A^{(0)}_{n}
\end{gathered}
\end{equation}
%
%
%
%
\subsubsection{The Gauss transformation.}\label{subsubsect:Gauss_Map}
Let $\mathrm{T}$ denote the Gauss transformation on the interval $[0,1]$, {\it i.e.\/},
\begin{equation}
\mathrm{T}(\theta)=
\left\{\begin{array}{ll}
\left\{\frac{1}{\theta}\right\} & \theta\in (0,1]\\
0 & \theta=0
\end{array}\right.
\end{equation}
where $\{x\}$ denotes the fractional part of the real number $x$.
\begin{comment}
Recall that the Gauss transformation possesses an 
ergodic absolutely continuous invariant 
probability measure $\mu$ given by
\begin{equation}
\mu=\frac{1}{\log 2}\frac{dx}{1+x}
\end{equation}
\end{comment}
Observe that $\mathrm{T}$ acts as a shift on the simple continued fraction representation of $\theta$.
Namely, take $\theta\in [0,1]$ and let $\hat\theta=\mathrm{T}(\theta)$. 
If the simple continued fraction expansion of $\theta$ is given by~\eqref{eq:theta-simple_ctd_frac_exp} above
then $\hat\theta$ has simple continued fraction expansion 
\begin{equation}
\hat\theta=[a_2,a_3,a_4,\ldots]
\end{equation}
For each non-negative integer $n$, let 
$p_n/q_n$
denote the $n$th convergent of $\theta$
and let 
$\hat{p}_n/\hat{q}_n$ 
denote the $n$th convergent of $\hat\theta$.
Observe that
\begin{equation}
\frac{p_n}{q_n}
=\frac{1}{a_1+[a_2,\ldots,a_n]}
=\frac{1}{a_1+\frac{\hat{p}_{n-1}}{\hat{q}_{n-1}}}
=\frac{\hat{q}_{n-1}}{a_1\hat{q}_{n-1}+\hat{p}_{n-1}}
\end{equation}
Recall that, by definition, $p_n$ and $q_n$ do not have common factors. 
However, a priori we do not know whether 
$\hat{q}_{n-1}$ 
and 
$a_1\hat{q}_{n-1}+\hat{p}_{n-1}$ 
have common factors or not.
(If not, we would have a direct relation between 
$p_n, q_n$ and $\hat{p}_{n-1}, \hat{p}_{n-1}$.)
\begin{theorem}\label{thm:p/q-action_under_Gauss_map}
Let $\theta\in [0,1]\setminus \mathbb{Q}$ and let $\hat\theta=\mathrm{T}(\theta)$.
Let $\theta$ and $\hat\theta$ have $n$th convergents $p_n/q_n$ and $\hat{p}_n/\hat{q}_n$ respectively.
Then, for all $n\in\mathbb{N}$,
\begin{equation}\label{eq:p/q-action_under_Gauss_map}
\left[\begin{array}{c}p_n\\ q_n\end{array}\right]
=
\left[\begin{array}{cc}0&1\\ 1&a_1\end{array}\right]
\left[\begin{array}{c}\hat{p}_{n-1}\\ \hat{q}_{n-1}\end{array}\right]
\end{equation}
\end{theorem}
\begin{proof}
Observe that if $\hat{q}_{n-1}$ and $a_1\hat{q}_{n-1}+\hat{p}_{n-1}$ have a common factor then for some positive integer $\lambda_n$ we have the following equalities
\begin{equation}
\label{eq:aux_expr3}
\left.\begin{array}{ll}
\lambda_n p_n = \hat{q}_{n-1}                   \quad & \lambda_{n+1} p_{n+1} = \hat{q}_n\\
\lambda_n q_n = a_1\hat{q}_{n-1}+\hat{p}_{n-1}  \quad & \lambda_{n+1} q_{n+1} = a_1\hat{q}_n+\hat{p}_n
\end{array}\right.
\end{equation}
But, by the recurrence relations~\eqref{eq:recurrence_rel_p} and~\eqref{eq:recurrence_rel_q}, and since $\hat{a}_n = a_{n+1}$, it follows that
\begin{align}
\begin{array}{lll}
p_{n+1} = a_{n+1}p_n+p_{n-1} \quad & \hat{p}_n = a_{n+1}\hat{p}_{n-1}+\hat{p}_{n-2}\\
q_{n+1} = a_{n+1}q_n+q_{n-1} \quad & \hat{q}_n = a_{n+1}\hat{q}_{n-1}+\hat{q}_{n-2}
\end{array}
\end{align}
Therefore 
\begin{equation}
\left.\begin{array}{ll}
\lambda_{n+1} \left(a_{n+1}p_n+p_{n-1}\right)&= a_{n+1}\hat{q}_{n-1}+\hat{q}_{n-2}\\
\lambda_{n+1} \left( a_{n+1}q_n+q_{n-1}\right)&=a_1\left(a_{n+1}\hat{q}_{n-1}+\hat{q}_{n-2}\right)+\left(a_{n+1}\hat{p}_{n-1}+\hat{p}_{n-2}\right)
\end{array}\right.
\end{equation}
Rearranging gives
\begin{align}
a_{n+1}\left(\lambda_{n+1}p_n-\hat{q}_{n-1}\right)+\left(\lambda_{n+1}p_{n-1}-\hat{q}_{n-2}\right)&=0\\
a_{n+1}\left(\lambda_{n+1}q_n-a_1\hat{q}_{n-1}-\hat{p}_{n-1}\right)+\left(\lambda_{n+1}q_{n-1}-a_1\hat{q}_{n-2}-\hat{p}_{n-2}\right)&=0
\end{align}
Then applying~\eqref{eq:aux_expr3}
\begin{equation}
\left.\begin{array}{ll}
a_{n+1}\left(\lambda_{n+1}p_n-\lambda_n p_n\right)+\left(\lambda_{n+1}p_{n-1}-\lambda_{n-1}p_{n-1}\right)&=0\\
a_{n+1}\left(\lambda_{n+1}q_n-\lambda_n q_n\right)+\left(\lambda_{n+1}q_{n-1}-\lambda_{n-1}q_{n-1}\right)&=0
\end{array}\right.
\end{equation}
Hence, rearranging once more,
\begin{equation}
\left.\begin{array}{ll}
p_{n-1}\left(\lambda_{n+1}-\lambda_{n-1}\right)
+a_{n+1}p_n\left(\lambda_{n+1}-\lambda_n\right)
&=0\\
q_{n-1}\left(\lambda_{n+1}-\lambda_{n-1}\right)
+a_{n+1}q_n\left(\lambda_{n+1}-\lambda_n\right)
&=0
\end{array}\right.
\end{equation}
In matrix form this can be expressed as
\begin{equation}
\left[\begin{array}{ll} p_{n-1} & p_{n}\\ q_{n-1} & q_{n}\end{array}\right]
\cdot
\left[\begin{array}{ll} 1 & 0\\ 0 & a_{n+1}\end{array}\right]
\cdot
\left[\begin{array}{l} \lambda_{n+1}-\lambda_n \\ \lambda_{n+1}-\lambda_{n-1}\end{array}\right] 
= 
\left[\begin{array}{l} 0 \\ 0\end{array}\right] 
\end{equation}
As $\theta$ is irrational, neither of the matrices on the left-hand side is singular.
Therefore, for all $n$,
\begin{equation}
\left[\begin{array}{l} \lambda_{n+1}-\lambda_n \\ \lambda_{n+1}-\lambda_{n-1}\end{array}\right] 
=
\left[\begin{array}{l}0\\ 0\end{array}\right]
\end{equation} 
{\it i.e.\/}, $\lambda_{n+1}=\lambda_n$ for all $n$.
But, by the recurrence relations~\eqref{eq:recurrence_rel_p} and~\eqref{eq:recurrence_rel_q}, 
we know that $\lambda_1=p_1/\hat{q}_0=1/1=1$, and the theorem follows.
\end{proof}
\noindent
An inductive argument, using the preceding result (Theorem~\ref{thm:p/q-action_under_Gauss_map}) 
together with the Binomial Theorem, now gives us the following.
\begin{corollary}
For any 
$r\in\mathbb{N}$ the following holds.
Let 
$\theta\in[0,1]\setminus\mathbb{Q}$ 
and 
$\hat\theta=\mathrm{T}(\theta)$.
Let 
$\theta$ and $\hat\theta$ 
have $n$th convergents 
$p_n/q_n$ and $\hat{p}_n/\hat{q}_n$ 
respectively.
Then for all $n\in\mathbb{N}$,
\begin{equation}
\left[\begin{array}{c}(p_n)^r(q_n)^0\\ (p_n)^{r-1}(q_n)^1\\ \vdots\\ (p_n)^1 (q_n)^{r-1}\\ (p_n)^0(q_n)^r\end{array}\right]
=
E(a_1;r)
\left[\begin{array}{c}(\hat{p}_{n-1})^r(\hat{q}_{n-1})^0\\ (\hat{p}_{n-1})^{r-1}(\hat{q}_{n-1})^1\\ \vdots\\ (\hat{p}_{n-1})^1(\hat{q}_{n-1})^{r-1}\\ (\hat{p}_n)^0(\hat{q}_{n-1})^r\end{array}\right]
\end{equation}
where
\begin{equation}
E(a_1;r)=
\left[\begin{array}{cccccc}
0     &0        &\cdots   &\cdots   &0   &1\\
0     &         &         &\revddots&1        &a_1\\ 
\vdots&         &\revddots&1        &2a_1     &a_1^2\\
\vdots&\revddots&\revddots&\revddots&         &\vdots\\
0     &1              &\binom{r-1}{1}a_1&\cdots & \binom{r-1}{r-2}a_1^{r-2} & a_1^{r-1}\\
1     &\binom{r}{1}a_1&\cdots         &\cdots & \binom{r}{r-1}a_1^{r-1}   & a_1^r
\end{array}\right]
\end{equation}
\end{corollary}
\begin{remark}
The matrix $E(a_1;r)$
can be factored as $E(a_1;r)=R(r)U(a_1;r)$ where $D(a_1;r)$ is the upper triangular matrix 
\begin{equation}
U(a_1;r)=
\left[\begin{array}{cccccc}
1     &\binom{r}{1}a_1&\cdots           &\cdots   & \binom{r}{r-1}a_1^{r-1}   & a_1^r\\
0     &1              &\binom{r-1}{1}a_1&\cdots   & \binom{r-1}{r-2}a_1^{r-2} & a_1^{r-1}\\
\vdots&\ddots         &\ddots           &\ddots   &                           &\vdots\\
\vdots&               &\ddots           &1        &2a_1                       &a_1^2\\
0     &               &                 &\ddots   &1                          &a_1\\ 
0     &0              &\cdots           &\cdots   &0                          &1\\
\end{array}\right]
\end{equation}
and $R(r)$ is the idempotent permutation matrix given by
\begin{equation}
R(r)=
\left[\begin{array}{ccccc}
0     &0        &\cdots   &0        &1\\
0     &         &\revddots &1        &0\\ 
\vdots&\revddots&\revddots&\revddots&\vdots\\
0     &1        &\revddots&         &0\\
1     &0        &\cdots   &0        &0
\end{array}\right]
\end{equation}  
Consequently $\det E(a_1;r)=\det R(r)=(-1)^{\sum_{2\leq n\leq r}(n-1)}$.
\end{remark}

We now consider iterating the action of the Gauss transformation.
Use the following notation.
For each positive integer $m$, 
let 
$\theta^{\caret m}=\mathrm{T}^m(\theta)$.
Then
\begin{equation}
\theta^{\caret m}
=[a_1^{\caret m},a_2^{\caret m},\ldots]
=[a_{m+1},a_{m+2},\ldots]
\end{equation}
Observe that 
$a^{\caret m}_n=a_{m+n}$, 
for each $m$ and $n$.
Denote the $n$th convergent of 
$\theta^{\caret m}$ 
by 
$p_n^{\caret m}/q_n^{\caret m}$.
Iterating the relation~\eqref{eq:p/q-action_under_Gauss_map}
and using $a^{\caret m}_n=a_{m+n}$, we find the following.

\begin{corollary}\label{cor:pn/qn-action_under_Gauss_map}
For each $r\in\mathbb{N}$ the following holds.
Given $\theta$ and $\hat{\theta}$ as above, for any non-negative integers $m,n$
\begin{equation}\label{eq:pn/qn-action_under_Gauss_map}
\left[\begin{array}{c}
(p_n^{\caret m})^r(q_n^{\caret m})^0\\
(p_n^{\caret m})^{r-1}(q_n^{\caret m})^1\\
\vdots\\
(p_n^{\caret m})^1(q_n^{\caret m})^{r-1}\\
(p_n^{\caret m})^0(q_n^{\caret m})^r\end{array}\right]
=
E
\left[\begin{array}{c}
(p_0^{\caret m+n})^r(q_0^{\caret m+n})^0\\
(p_0^{\caret m+n})^{r-1}(q_0^{\caret m+n})^1\\
\vdots\\
(p_0^{\caret m+n})^{1}(q_0^{\caret m+n})^{r-1}\\ 
(p_0^{\caret m+n})^0(q_0^{\caret m+n})^r
\end{array}\right]
\end{equation}
where $E=E(a_{m+1};r)E(a_{m+2};r)\cdots E(a_{m+n};r)$.
\end{corollary}
\begin{remark}\label{rmk:pn/qn-action_under_Gauss_map}
In particular, for $r=1$ and positive integers $m$ and $n$
\begin{equation}\label{eq:pn/qn-action_under_Gauss_map-2d}
\left[\begin{array}{c}p_n^{\caret m}\\ q_n^{\caret m}\end{array}\right]
=
\left[\begin{array}{cc}0&1\\ 1&a_{m+1}\end{array}\right]
\left[\begin{array}{cc}0&1\\ 1&a_{m+2}\end{array}\right]
\cdots
\left[\begin{array}{cc}0&1\\ 1&a_{m+n}\end{array}\right]
\left[\begin{array}{c}p^{\caret m+n}_0\\ q^{\caret m+n}_0\end{array}\right]
\end{equation}
\end{remark}
\vspace{5pt}

%
%
%
%
\subsubsection{Quadratic Irrationals.}\label{subsubsect:Quad_Irr}
Let $\theta\in[0,1]\setminus\mathbb{Q}$ 
be a quadratic irrational.
By this we will mean that $\theta$ is an 
algebraic number whose minimal polynomial 
$\omega_\theta$ is of (strict) degree two.
Then $\theta$ possesses a unique Galois conjugate which we denote by $\theta'$.
%
\begin{comment}
\begin{claim}
If quadratic irrational has periodic continued fraction expansion
\begin{equation}
\left[\overline{a_1,a_2,\ldots,a_\mathfrak{\ell}}\right]
\end{equation}
then $\theta'$ has continued fraction expansion
\begin{equation}
\left[\overline{a_\ell,a_{\ell-1},\ldots,a_\mathfrak{1}}\right]
\end{equation}
\end{claim}
\end{comment}

A theorem of Lagrange~\cite[p.56]{KhinchinBook} 
states that $\theta$ has a pre-periodic simple continued fraction expansion, 
{\it i.e.}, there exists a finite sequence of positive integers 
$a_1,a_2,\ldots,a_\kay,\ldots,a_{\kay+\ell}$ 
such that
\begin{equation}
\theta=[a_1,a_2,\ldots,a_{\kay},\overline{a_{\kay+1},\ldots,a_{\kay+\ell}}]
\end{equation}
(We adopt the convention that $\kay=0$ actually means the continued fraction expansion is periodic).
We call the minimal such $\ell$ the {\it period}. 
We call any such $\kay$ a {\it preperiod} and the least such preperiod the 
{\it minimal preperiod} of the simple continued fraction expansion.
\begin{remark}
Since $a_{n+\ell}=a_{n}$ for all $n>\kay$, by~\eqref{eq:AB_definition} it follows 
that, for all non-negative integers $m$ and $n$ satisfying $n-m>\kay$, we have
\begin{equation}
A^{(m)}_{n-m+\ell-1}=A^{(m)}_{n-m-1}, \qquad B^{(m)}_{n-m+\ell-1}=B^{(m)}_{n-m-1}
\end{equation}
\begin{comment}
This follows as
\begin{align}
\left[\begin{array}{cc}B^{(m)}_{n-m-1}&B^{(m+1)}_{n-m-1}\\ A^{(m)}_{n-m-1}&A^{(m)}_{n-m-1}\end{array}\right]
&=
\left[\begin{array}{cc}0&1\\ 1&a_{n-m}\end{array}\right]
\left[\begin{array}{cc}0&1\\ 1&a_{n-m+1}\end{array}\right]
\cdots
\left[\begin{array}{cc}0&1\\ 1&a_{n}\end{array}\right]\\
&=
\left[\begin{array}{cc}0&1\\ 1&a_{n-m+\ell}\end{array}\right]
\left[\begin{array}{cc}0&1\\ 1&a_{n-m+1+\ell}\end{array}\right]
\cdots
\left[\begin{array}{cc}0&1\\ 1&a_{n+\ell}\end{array}\right]\\
&=
\left[\begin{array}{cc}B^{(m)}_{n+\ell-m-1}&B^{(m+1)}_{n+\ell-m-1}\\ A^{(m)}_{n+\ell-m-1}&A^{(m)}_{n+\ell-m-1}\end{array}\right]
\end{align}
\end{comment}
In particular, the following quantities are well-defined
\begin{equation}\label{eq:AB_periodic}
A^{(m)}_{(j)}=A^{(m)}_{\kay+n}, \qquad B^{(m)}_{(j)}=B^{(m)}_{\kay+n} \qquad \mbox{for any} \ \ n\geq 0, \ \ \kay+n=j\mod \ell
\end{equation}
\end{remark}
Recall that we defined
$\theta^{\caret m}=\mathrm{T}^{m}(\theta)$ for $m\geq 0$. 
In particular, 
$\theta^{\caret \kay}=\mathrm{T}^{\kay}(\theta)$. 
Hence 
$\mathrm{T}^\ell(\theta^{\caret \kay})=\theta^{\caret \kay}$.
Thus 
$\theta^{\caret \kay}$ 
is a solution to the equation
\begin{equation}
\theta^{\caret \kay}=\frac{1}{a_{\kay+1}+}\frac{1}{a_{\kay+2}+}\cdots\frac{1}{a_{\kay+\ell}+\theta^{\caret \kay}}
\end{equation}
and $\theta$ can be expressed as
\begin{equation}
\theta=\frac{1}{a_{1}+}\frac{1}{a_{2}+}\cdots\frac{1}{a_{\kay}+\theta^{\caret \kay}}
\end{equation}
These can be written in matrix form as
\begin{equation}
\left[\begin{array}{c}\theta^{\caret \kay}\\ 1\end{array}\right]
=
N_1
\left[\begin{array}{c}\theta^{\caret \kay}\\ 1\end{array}\right],
\qquad
\left[\begin{array}{c}\theta\\ 1\end{array}\right]
=
N_0
\left[\begin{array}{c}\theta^{\caret \kay}\\ 1\end{array}\right]
\end{equation}
where
\begin{equation}\label{eq:a1a2...al_1}
\begin{gathered}
N_1
=
\left[\begin{array}{cc}0&1\\ 1&a_{\kay+1}\end{array}\right]
\left[\begin{array}{cc}0&1\\ 1&a_{\kay+2}\end{array}\right]
\cdots
\left[\begin{array}{cc}0&1\\ 1&a_{\kay+\ell}\end{array}\right]\\ \vspace{5pt}
N_0
=
\left[\begin{array}{cc}0&1\\ 1&a_{1}\end{array}\right]
\left[\begin{array}{cc}0&1\\ 1&a_{2}\end{array}\right]
\cdots
\left[\begin{array}{cc}0&1\\ 1&a_{\kay}\end{array}\right]
\end{gathered}
\end{equation}
Here, as usual, we identify the matrices $N_0$ and $N_1$ with their corresponding elements in  
$\mathrm{PGL}(2,\mathbb{Z})$.
\begin{comment}
\begin{lemma}\label{lem:q_in_terms_of_A}
In the case when 
$\theta\in [0,1]\setminus\mathbb{Q}$ 
has periodic continued fraction expansion of period $\ell$,
we have $q_n=A^{(n)}_\ell$ and $p_n=A^{(n-1)}_1$ for all $n$.
\end{lemma}
%
\begin{proof}
We give a proof by induction.
First, by the recurrence relations~\eqref{eq:recurrence_rel_q} 
we find that
\begin{align}
q_1&=a_1 &\qquad q_2&=a_2q_1+q_0=a_2a_1+1\\
p_1&=1 &\qquad p_2&=a_2p_1+p_0=a_2 
\end{align}
By 
equation~\eqref{eq:AB_definition} 
and~\eqref{eq:AB_recurrence_rel} 
we have 
\begin{align}
A^{(1)}_\ell&=a_1 &\qquad A^{(2)}_\ell&=1+a_1a_2\\
A^{(0)}_1   &=1   &\qquad A^{(1)}_1&=a_2
\end{align}
Therefore the lemma holds for $n=1,2$.
Next, assume the hypothesis holds for all positive $k<n$.
The recurrence relations~\eqref{eq:recurrence_rel_q} and~\eqref{eq:AB_splitting_rel_2}, 
together with the assumption that $a_{n+\ell}=a_n$ for all $n$ and the induction hypothesis, implies that
\begin{equation}
q_n
=a_nq_{n-1}+q_{n-2}
=a_{n+\ell}A^{(n-1)}_\ell+A^{(n-2)}_\ell
=A^{(n)}_\ell
\end{equation}
An analogous computation shows $p_n=A^{(n-1)}_1$.
Thus $q_n=A^{(n)}_\ell$ and $p_n=A^{(n-1)}_{1}$, and by induction this identity holds for all $n$.
\end{proof}
\end{comment}
%
From equation~\eqref{eq:AB_definition}, equation~\eqref{eq:AB_recurrence_rel}, and applying definition~\eqref{eq:AB_periodic} we have
\begin{equation}\label{eq:a1a2...al_2}
N_1
=
\left[\begin{array}{cc}
B^{(\ell-1)}_{(\kay)} & B^{(\ell)}_{(\kay)}\\A^{(\ell-1)}_{(\kay)} & A^{(\ell)}_{(\kay)}
\end{array}\right]
=
\left[\begin{array}{cc}
A^{(\ell-2)}_{(\kay+1)} & A^{(\ell-1)}_{(\kay+1)}\\A^{(\ell-1)}_{(\kay)} & A^{(\ell)}_{(\kay)}
\end{array}\right]
\end{equation}
and
\begin{equation}
N_0
=
\left[\begin{array}{cc}
B^{(\kay-1)}_{(0)} & B^{(\kay)}_{(0)}\\A^{(\kay-1)}_{(0)} & A^{(\kay)}_{(0)}
\end{array}\right]
=
\left[\begin{array}{cc}
A^{(\kay-2)}_{(1)} & A^{(\kay-1)}_{(1)}\\A^{(\kay-1)}_{(0)} & A^{(\kay)}_{(0)}
\end{array}\right]
\end{equation}
Hence $\theta^{\caret\kay}$ has minimal polynomial
\begin{equation}\label{eq:min_poly}
\omega_{\theta^{\caret\kay}}(z)
=A^{(\ell-1)}_{(\kay)} z^2 
+ \left(A^{(\ell)}_{(\kay)}-A^{(\ell-2)}_{(\kay+1)}\right)z 
- A^{(\ell-1)}_{(\kay+1)} 
\end{equation}
and, since $\theta=N_0(\theta^{\caret\kay})$ (now viewing $N_0$ as a linear fractional transformation), 
it follows that $\theta$ has minimal polynomial
\begin{equation}
\omega_\theta(z)
=
\left(A^{(\kay-2)}_{(1)}-A^{(\kay-1)}_{(0)}z\right)^2
\omega_{\theta^{\caret\kay}}(N_0^{-1}(z))
\end{equation}
(Note that the above shows the minimal 
polynomial of $\theta$ has degree two 
or less. However, since $\theta$ has 
non-terminating continued fraction 
expansion, we know it cannot be rational, 
{\it i.e.\/}, cannot have minimal 
polynomial of degree one.)

\subsection{The projective general linear group}\label{subsect:PGL}
We collect here some properties of 
the two-dimensional real projective general linear group and its subgroups.
We recommend the reader consult~\cite{Sarnak1982,Series1984,Pollicott1986,Faivre1992}.
Here, our presentation is slightly different as 
the real projective general linear group
is more natural than 
the real projective special linear group
when we consider simple continued fraction expansions (cf.~\cite{Sarnak2007}).

Throughout, 
$\mathbb{H}^+$ and $\mathbb{H}^-$ 
denote respectively the upper and lower half-planes of $\mathbb{C}$, 
both with the orientations induced by the natural embeddings into $\mathbb{C}$.
We endow both half-planes with the Poincar\'e metric $(dx^2+dy^2)/y^2$. 
Geodesics with respect to this metric are half-circles and straight-lines 
perpendicular to the real axis.
The (limiting) intersection points of these geodesics with the real axis are called the {\it ends}. 
Given two distinct points $a,b\in\mathbb{H}^\pm$ we denote 
the unique oriented hyperbolic geodesic arc from $a$ to $b$ by $[a,b]$.

\begin{comment}
Let 
\begin{align}
\mathrm{GL}(2,\mathbb{R})
&=\left\{ A\in\mathrm{Mat}_{2\times 2}(\mathbb{R}) : \det(A)=\pm 1\right\}\\
\mathrm{SL}(2,\mathbb{R})
&=\left\{ A\in\mathrm{Mat}_{2\times 2}(\mathbb{R}) : \det(A)=1\right\}
\end{align}
Trivially 
$\mathrm{SL}(2,\mathbb{R})$ 
is a normal subgroup of 
$\mathrm{GL}(2,\mathbb{R})$.
In fact
\begin{equation}
\mathrm{GL}(2,\mathbb{R})
=\mathrm{SL}(2,\mathbb{R})\cup E\cdot \mathrm{SL}(2,\mathbb{R}) 
\end{equation}
where 
$E=\left[\begin{array}{cc}1&0\\0&-1\end{array}\right]$.
Thus
$\mathrm{SL}(2,\mathbb{R})$ 
is an index 2 subgroup of 
$\mathrm{GL}(2,\mathbb{R})$.
(Recall, more generally, that all index two subgroups of a given group $G$ are normal.)
\end{comment}
Denote by
$\mathrm{PGL}(2,\mathbb{R})$ 
and 
$\mathrm{PSL}(2,\mathbb{R})$ 
respectively 
the 2-dimensional real projective general linear group and
the 2-dimensional real projective special linear group.
Then $\mathrm{PSL}(2,\mathbb{R})$ is a subgroup of $\mathrm{PGL}(2,\mathbb{R})$ of index 2 and,
letting $R$ denote the element of $\mathrm{PGL}(2,\mathbb{R})$ of order 2 represented by the matrix
$\left[\begin{array}{cc}1&0\\0&-1\end{array}\right]$,
we have $\mathrm{PSL}(2,\mathbb{R})\cong\mathrm{PGL}(2,\mathbb{R}) / \langle R\rangle$.
Given 
a subgroup $H$ of $\mathrm{PGL}(2,\mathbb{R})$ and
an element $S\in\mathrm{PGL}(2,\mathbb{Z})$
let 
$[S]_{H}=\{TST^{-1} : T\in H\}$. 
In particular,
\begin{equation}
[S]_{\mathrm{PGL}(2,\mathbb{Z})}=[S]_{\mathrm{PSL}(2,\mathbb{Z})}\cup R\cdot[S]_{\mathrm{PSL}(2,\mathbb{Z})}\cdot R
\end{equation}

The group
$\mathrm{PGL}(2,\mathbb{R})$ is isomorphic to 
the group of linear fractional transformations acting on $\mathbb{C}$ which preserve the real line $\mathbb{R}$,
and 
the subgroup
$\mathrm{PSL}(2,\mathbb{R})$ is isomorphic to 
the subgroup of 
linear fractional transformations acting on $\mathbb{C}$ which preserve the upper half-plane $\mathbb{H}^+$ 
(and consequently also the lower half-plane $\mathbb{H}^-$).
The left-coset $R\cdot \mathrm{PSL}(2,\mathbb{R})$ consists of linear fractional transformations
interchanging $\mathbb{H}^+$ and $\mathbb{H}^-$.
Thus 
$\mathrm{PGL}(2,\mathbb{R})$ 
acts on 
$\mathbb{H}^-\cup\mathbb{H}^+$, 
while the action of
$\mathrm{PSL}(2,\mathbb{R})$ 
restricts separately to $\mathbb{H}^-$ and to $\mathbb{H}^+$. 
Similar statements hold for appropriately chosen subgroups of 
$\mathrm{PGL}(2,\mathbb{R})$ 
and 
$\mathrm{PSL}(2,\mathbb{R})$, 
such as
$\mathrm{PGL}(2,\mathbb{Z})$ 
and 
$\mathrm{PSL}(2,\mathbb{Z})$. 
Since, for any group $G$ acting on a some set $X$ and has finite subgroup $H$ we have $(X/H)/(G/H)\cong X/G$, 
we therefore get the following, 
%
\begin{equation}
\mathbb{H}^+/\mathrm{PSL}(2,\mathbb{Z})\cong (\mathbb{H}^-\cup\mathbb{H}^+)/\mathrm{PGL}(2,\mathbb{Z})
\end{equation}
The above surface is the {\it modular surface}, which we henceforth denote by $\mathcal{M}$.
It can be represented as a quotient of $\mathbb{H}^+$ and also of $\mathbb{H}^-$.
Denote by $\pi^\pm\colon \mathbb{H}^\pm\to\mathcal{M}$ the corresponding canonical projections.
We will also consider the double cover
\begin{equation}
(\mathbb{H}^-\cup\mathbb{H}^+)/\mathrm{PSL}(2,\mathbb{Z})
\end{equation}
The above surface we refer to as the {\it double of the modular surface}, denoted by $\mathcal{M}_2$.
This consists of two connected components, $\mathcal{M}_2^-$ and $\mathcal{M}_2^+$, both isomorphic to $\mathcal{M}$.  
In fact, there are two isomorphisms: one orientation-preserving, the other orientation-reversing. 
The orientation-preserving isomorphism is induced by $R$.
The orientation-reversing isomorphism is induced by complex conjugation.
Observe that both $R$ and complex conjugation are idempotents which descend to idempotents interchanging $\mathcal{M}_2^-$ and $\mathcal{M}_2^+$.
We will refer to the isomorphism induced by complex conjugation as {\it complex conjugation on $\mathcal{M}_2$},
and we will call sets 
$S^-\subset \mathcal{M}_2^-$ and 
$S^+\subset\mathcal{M}_2^+$
{\it conjugate-related} if this isomorphism interchanges $S^-$ and $S^+$.

Recall that the half-planes $\mathbb{H}^-$ and $\mathbb{H}^+$ 
can both be endowed with the Poincar\'e metric 
$(dx^2+dy^2)/y^2$.
This metric is invariant under the action of 
$\mathrm{PGL}(2,\mathbb{R})$.
(Since the subgroups $\langle R\rangle$ and $\mathrm{PSL}(2,\mathbb{R})$ both leave the metric invariant.)
Therefore the Poincar\'e metric descends to a metric on the modular surface, 
which we call the {\it Poincar\'e metric on the modular surface}, or just the {\it Poincar\'e metric} when there is no possible ambiguity.

The modular surface is a non-compact hyperbolic Riemann surface of finite area, with respect to the volume form induced by the Poincar\'e metric,
and has one puncture and two ramification points (see, {\it e.g.},~\cite{GunningBook1962}).

\subsubsection{Hyperbolic elements and geodesics}\label{subsubsect:hyp_elements}
An element of
$\mathrm{PGL}(2,\mathbb{R})$ 
is {\it hyperbolic} if its action on 
$\mathbb{C}$ 
possesses two (distinct) fixed points, both of which 
are contained in the extended real line. 
\begin{remark}\label{rmk:hyp_conjugate_to_diagonal}
The element 
$M\in\mathrm{PGL}(2,\mathbb{R})$ 
is hyperbolic if and only if it is conjugate, 
via an element of 
$\mathrm{PSL}(2,\mathbb{R})$, 
either to an element of the form 
$\left[\begin{array}{cc}t&0\\0&t^{-1}\end{array}\right]$ 
(if $M\in \mathrm{PSL}(2,\mathbb{R})$), 
or an element of the form
$\left[\begin{array}{cc}t&0\\0&-t^{-1}\end{array}\right]$ 
(if $M\in R\cdot\mathrm{PSL}(2,\mathbb{R})$), 
where, in both cases, $t$ is real with $t>1$.
\end{remark}
Let 
$M\in \mathrm{PGL}(2,\mathbb{R})$ 
be hyperbolic, with fixed points 
$\theta_+$ and $\theta_-$.
Let 
$\gamma_M$ 
denote the unique circle in $\mathbb{C}$ perpendicular to 
the extended real axis passing through $\theta_+$ and $\theta_-$.
Observe that
$M$ preserves $\gamma_M$.
Let 
$\gamma_M^\pm=\gamma_M\cap\mathbb{H}^\pm$. 
Observe further that complex conjugation interchanges $\gamma_M^-$ and $\gamma_M^+$.
Endow 
$\gamma_M^+$ 
with an arbitrary orientation and give 
$\gamma_M^-$ 
the orientation induced by complex conjugation.

Consider the quotients of 
$\mathbb{H}^-\cup\mathbb{H}^+$, 
and the corresponding quotients of the subset $\gamma_M$, 
by the discrete subgroups 
$\mathrm{PSL}(2,\mathbb{Z})$ and
$\mathrm{PGL}(2,\mathbb{Z})$.
The oriented hyperbolic geodesics $\gamma^+_M$ and $\gamma^-_M$ descend 
to conjugate-related oriented hyperbolic geodesics on $\mathcal{M}_2$ (the orientations chosen above were made so we had agreement here)
\footnote{Thus $\gamma_M^+$ and $\gamma_M^-$ descend
to a pair of geodesics on $\mathcal{M}$.
The geodesics $\gamma_M^+$ and $\gamma_M^-$ 
descend to the same geodesic on $\mathcal{M}$ 
if and only if 
$\gamma_M$ is preserved by some non-trivial 
element of $\mathrm{PGL}(2,\mathbb{Z})$ of negative determinant, 
necessarily interchanging $\gamma_M^-$ and $\gamma_M^+$.
See~\cite{Sarnak2007} for more information on when this occurs.}.
Either of these geodesics (and hence both) descends 
to a closed geodesic on $\mathcal{M}_2$ 
if and only if 
$\gamma_M$ is preserved by some hyperbolic element $N$ of 
$\mathrm{PGL}(2,\mathbb{Z})$
.

If 
$N\in\mathrm{PGL}(2,\mathbb{R})$ 
is conjugate to $M$ by some element $U$ of 
$\mathrm{PGL}(2,\mathbb{Z})$
then either
$\gamma_{N}^+ =U(\gamma_M^+)$ 
or
$\gamma_{N}^+ =U(\gamma_M^-)$. 
Thus the (unordered) pair of geodesics
$\gamma_{N}^+$ and $\gamma_{N}^-$ 
descend to the same (unordered) pair of conjugate-related hyperbolic geodesics 
in 
$\mathcal{M}_2$ as the (unordered) pair $\gamma_{M}^+$ and $\gamma_{M}^-$.
Conversely, 
given a conjugate-related pair of closed geodesics
$\gamma^-$ and $\gamma^+$ 
on $\mathcal{M}_2$, a lift of $\gamma^+$ induces a lift of $\gamma^-$ (and vice-versa) 
and any two different lifts of $\gamma^\pm$ are related via an element $U$ of $\mathrm{PGL}(2,\mathbb{Z})$.
Geodesically completing the lifts of $\gamma^-$ and $\gamma^+$ gives a circle perpendicular to the extended real axis or vertical straight line, 
and hence a hyperbolic element $M$ of $\mathrm{PGL}(2,\mathbb{R})$. 
Any other hyperbolic element $N$ constructed in this way, from another lift, 
will be related via some $U$ to $M$, {\it i.e.} $M=UNU^{-1}$.  
Hence we get the following variant of a well-known result:
\begin{proposition}
Conjugacy classes 
$[M]_{\mathrm{PGL}(2,\mathbb{Z})}$ 
of hyperbolic elements $M$ in 
$\mathrm{PGL}(2,\mathbb{Z})$ 
are in bijective correspondence with conjugate-related 
pairs of oriented closed geodesics on 
$\mathcal{M}_2$.
\end{proposition}

\subsubsection{Prime hyperbolic elements and prime geodesics}\label{subsubsect:prime_hyp}
The hyperbolic element 
$M\in\mathrm{PGL}(2,\mathbb{Z})$ 
is {\it prime} if given 
any point $p$ in $\gamma_M^\pm$, 
the image $M(p)$ also lies in $\gamma_M^\pm$, 
and the geodesic segment $[p,M(p)]$ of $\gamma_M^\pm$  
is a fundamental domain
for the action of the subgroup 
\begin{equation}
\mathfrak{S}_\mathbb{Z}
=
\left\{N\in\mathrm{PSL}(2,\mathbb{Z}) : N(\gamma_M^\pm)=\gamma_M^\pm , \ N(\theta_\pm)=\theta_\pm\right\}
\end{equation}
on $\gamma_M^\pm$.
Observe that prime elements must lie in 
$\mathrm{PSL}(2,\mathbb{Z})$. 
The property of being prime is therefore equivalent 
to requiring that $M$ is not a non-trivial power of 
some other hyperbolic element in $\mathrm{PSL}(2,\mathbb{Z})$.
(This is the more usual definition when considering the action 
on $\mathbb{H}^+$ of the group $\mathrm{PSL}(2,\mathbb{Z})$, cf.~\cite{Sarnak1982}).
\begin{comment}
If $M=N^p$ where $N$ is hyperbolic, then $\mathrm{Fix}(N)=\mathrm{Fix}(M)$.
(If not, $M$ would have more than two real fixed-points.)
\end{comment}
Consider the real one-parameter subgroup of $\mathrm{PGL}(2,\mathbb{R})$ given by
\begin{equation}
\mathfrak{S}_\mathbb{R}
=
\left\{N\in\mathrm{PSL}(2,\mathbb{R}) : N(\gamma_M^\pm)=\gamma_M^\pm , \ N(\theta_\pm)=\theta_\pm\right\}
\end{equation}
Since the action of $\mathrm{PGL}(2,\mathbb{Z})$ is discrete, 
we find that 
$\mathfrak{S}_{\mathbb{Z}}=\mathfrak{S}_{\mathbb{R}}\cap\mathrm{PGL}(2,\mathbb{Z})$ 
also acts on $\gamma_M$ discretely and, in fact, is an infinite cyclic subgroup. 
Hence the notion of a prime hyperbolic element is well-defined.
Moreover, 
given any hyperbolic element $M\in\mathrm{PGL}(2,\mathbb{Z})$
there will be exactly $2$ prime hyperbolic elements preserving $\gamma_M$, 
these two elements will be mutually inverse, and either of these $2$ elements generates $\mathfrak{S}_\mathbb{Z}$.
When $\gamma_M^\pm$ is oriented, we will take the prime hyperbolic element for which the orientation of $[p,M(p)]$ agrees with the orientation on $\gamma_M$.

A closed geodesic $\gamma$ on $\mathcal{M}_2$ 
is {\it prime} if there does not exist another 
geodesic which traces out the same set of points as $\gamma$ on $\mathcal{M}_2$, 
and of strictly smaller length.
Observing that if $M$ is prime hyperbolic the $[p,M(p)]$ and $[\bar{p},M(\bar{p})]$ 
descend to conjugate-related pairs of prime geodesics on $\mathcal{M}_2$, together with the argument from the preceding Section~\ref{subsubsect:hyp_elements}
gives us the following variant of another well-known result:
\begin{proposition}
Conjugacy classes 
$[M]_{\mathrm{PGL}(2,\mathbb{Z})}$ 
of prime hyperbolic elements $M$ in 
$\mathrm{PGL}(2,\mathbb{Z})$ 
are in bijective correspondence with conjugate-related 
pairs of prime oriented closed geodesics on 
$\mathcal{M}_2$.
\end{proposition}
Consider Remark~\eqref{rmk:hyp_conjugate_to_diagonal} above. 
Given a prime hyperbolic element $M$ define the {\it norm} of $M$ by
$\mathrm{n}(M)=t^2$
Observe that this can be computed explicitly from the trace $\trace (M)=t+t^{-1}$, and hence only depend on the conjugacy class of $M$.
Recalling that the hyperbolic distance between the points $\imath t_1$ and $\imath t_2$ on the imaginary axis is given by $|\log (t_2/t_1)|$,
we get the following, also well-known, result:
\begin{corollary}\label{cor:prime-length_of_geodesics} 
Let $\gamma^-$ and $\gamma^+$ be a conjugate-related pair of prime oriented closed geodesics in $\mathcal{M}_2$.
Let $[M]_{\mathrm{PGL}(2,\mathbb{Z})}$ denote the corresponding prime hyperbolic conjugacy class.
Then 
$
\mathrm{length}(\gamma^-)
=\mathrm{length}(\gamma^+)
=|\log n(M)|
$.
\end{corollary}
\subsubsection{Prime hyperbolic elements and algebraic real numbers of degree two}\label{subsubsect:hyp_elements+quad_irr}
We now recall that prime hyperbolic 
elements are in bijective correspondence with Galois conjugate 
pairs of algebraic real numbers of degree two. 
%
Namely, given any hyperbolic element 
$M=\left[\begin{array}{cc}a&b\\c&d\end{array}\right]$ in $\mathrm{PGL}(2,\mathbb{Z})$, 
the fixed point equation 
for $M$ 
may be rearranged to give a degree-two polynomial with integer coefficients with positive discriminant.
Therefore the fixed points of $M$ form a Galois conjugate pair of algebraic reals of degree two.
(In the degenerate case when $c=0$ we take one of the pair to be infinity.)

In the opposite direction, 
degree-two algebraic real numbers are in one-to-one correspondence with pre-periodic simple continued fraction expansions.
Let $\theta$ be a degree-two algebraic real number with Galois conjugate $\theta'$.
The case when $\theta$ has strict degree one 
(and hence is rational, with finite continued fraction expansion) is straightforward.
Therefore assume that $\theta$ has strict degree two with
simple continued fraction expansion
\begin{equation}
[a_1,a_2,\ldots,a_\kay,a_{\kay+1},\ldots,a_{\kay+\ell}]
\end{equation}
Define
\begin{equation}
N_0=E(a_1)E(a_2)\cdots E(a_{\kay}) \qquad
N_1=E(a_{\kay+1})E(a_{\kay+2})\cdots E(a_{\kay+\ell})
\end{equation}
where,
for each $a\in\mathbb{N}$, 
$E(a)=\left[\begin{array}{cc}0&1\\1&a\end{array}\right]$.
(These agree with the matrices defined by equation~\eqref{eq:a1a2...al_1} in Section~\ref{subsubsect:Quad_Irr}.)
Let 
$N=N_\theta=N_0N_1 N_0^{-1}$.
By considering the trace or otherwise we find that $N_1$, and hence $N$, 
is hyperbolic. 
Moreover, the hyperbolic element $N_1$, and hence $N$, is prime if and only if 
$\ell$ is even and either 
\begin{itemize}
\item[(a)]
is the minimal period of  
$a_{\kay+1},a_{\kay+2},\ldots,a_{\kay+\ell}$, or
\item[(b)] 
is twice the minimal period of 
$a_{\kay+1},a_{\kay+2},\ldots,a_{\kay+\ell}$, which is odd.
\end{itemize}
Next, observe that
$N(\theta)=\theta$. 
Thus
$\theta$ is a solution of a quadratic equation $Q_N(z)=0$, 
where $Q_N$ has coefficients in $\mathbb{Z}$. 
Since $\theta$ is real,  
$\discr(\chi_N)=\discr(\chi_{N_1})\geq 0$, where $\chi_N$ denotes the characteristic polynomial of $N$.
Since we already know that $\theta$ has degree $2$ minimal polynomial (and hence $N_1$ is non-degenerate) 
we must have 
$Q_N(z)=\kappa\cdot\omega_\theta(z)$, 
for some 
$\kappa\in\mathbb{Q}$. 
Thus, $N$ must also fix the Galois conjugate $\theta'$ of $\theta$.
Hence to each Galois conjugate pair of degree two algebraic real numbers $\theta$ and $\theta'$, 
there exists $N\in\mathrm{PGL}(2,\mathbb{Z})$ which is prime and fixes $\theta$ and $\theta'$.
%
%

\subsection{Dynamical zeta functions}\label{subsect:dyn_zeta}
Our aim in this section is to consider the 
dynamical zeta function $\zeta_f$ of the 
hyperbolic toral automorphism $f$.
Recall that the dynamical zeta function of 
a general discrete-time dynamical system $f$ with weight $g$ is given by
\begin{equation}
\zeta_{f,g}(z)
=\exp\left(
\sum_{n\geq 1}\frac{z^n}{n}
\sum_{x\in\mathrm{Fix} f^n}\prod_{m=1}^{n}g(f^mx)
\right)
\end{equation}
When the weight is positive the radius of convergence $\rho$ is given by 
$\rho=\exp(-P(\log g))$ 
where $P(\log g)$ is the {\it pressure} of the potential $\log g$ given by
\begin{equation}
P(\log g)
=\limsup_{n\to\infty}\frac{1}{n}\log\left(
\sum_{x\in\mathrm{Fix} f^n}\prod_{m=1}^{n}g(f^mx)
\right)
\end{equation}
For the weight $g\equiv 1$ we get the Artin-Mazur dynamical zeta function, which we denote by $\zeta_f$.
Then by the Cauchy-Hadamard theorem 
\begin{equation}
1/\rho
=\limsup_{n\to\infty}|\card\mathrm{Fix}(f^n)|^{1/n}
\end{equation}
By considering the Taylor expansion of $\log(1-z^p)$ the dynamical zeta function of weight $g$ can be expressed as an Euler product
\begin{equation}\label{eq:zeta-map-g_arb-euler_form}
\zeta_{f,g}(z)
=\prod_{\varpi \ \mbox{\tiny prime}} 
\left(1-z^{\mathrm{per}(\varpi)}\prod_{m=0,\ldots,\mathrm{per}(\varpi)-1}g(f^m x(\varpi))\right)^{-1}
\end{equation}
where the product is taken over all prime periodic orbits $\varpi$ of $f$, 
$\mathrm{per}(\varpi)$ denotes the prime period of $\varpi$ and 
$x(\varpi)$ denotes an arbitrary element of the orbit $\varpi$.
For $g\equiv 1$ reduces to
\begin{equation}\label{eq:zeta-map-g=1-euler_form}
\zeta_{f,g}(z)
=\prod_{\varpi \ \mbox{\tiny prime}}
\left(1-z^{\mathrm{per}(\varpi)}\right)^{-1}
\end{equation}
When considering, instead, a continuous-time dynamical system $f$ with weight $g$ the dynamical zeta function is given by
\begin{equation}
\zeta_{f,g}(s)=\prod_{\varpi}\left(1-\exp\left[-s\int_0^{\mathrm{per}(\varpi)} g\left(f^t x(\varpi)\right) \, dt\right]\right)^{-1}
\end{equation}
where, as before, the product is taken over all prime periodic orbits $\varpi$ of $f$, 
$\mathrm{per}(\varpi)$ denotes the prime period of $\varpi$ and 
$x(\varpi)$ denotes an (arbitrarily chosen) point of $\varpi$.
The (formal) equivalence between this expression and the resulting time-one map is via an Euler product formula-type argument.
In the special case when $g\equiv 1$ the above expression reduces to
\begin{equation}\label{eq:zeta-flow-g=1}
\zeta_{f,g}(s)=\prod_{\varpi}\left(1-\exp\left[-s\cdot \mathrm{per}(\varpi)\right]\right)^{-1}
\end{equation}

\subsection{Hyperbolic toral automorphisms.}\label{subsect:entropy+toral_automorphism}
First, let us recall the following facts concerning general toral automorphisms.
That is,
maps of the form
\begin{equation}
f\colon \mathbb{R}^d/\mathbb{Z}^d\to\mathbb{R}^d/\mathbb{Z}^d \qquad
f(x+\mathbb{Z}^d)=M_fx+\mathbb{Z}^d
\end{equation}
where $M_f\colon\mathbb{R}^d\to\mathbb{R}^d$ is a linear map with integer entries satisfying
$\det(M_f)=\pm 1$.
It is known that 
$f$ is ergodic if and only if $\spec(M_f)$ 
does not contain any roots of unity~\cite[p.31]{WaltersBook}.
Since, for arbitrary endomorphisms of compact groups, 
ergodicity is equivalent to weak mixing is equivalent to strong mixing~\cite[p.50]{WaltersBook},
we also find that $f$ is strong mixing provided 
$\spec(M_f)$ doesn't contain any roots of unity.
Finally, $f$ is expansive if and only if it is hyperbolic, 
{\it i.e.}, 
$\spec(M_f)$ is disjoint from the unit circle~\cite[p.143]{WaltersBook}.
(However, expansivity may be lost when taking finite-to-one quotients~\cite[Example 1, p.140]{WaltersBook}.)

Observe that, for each integer $n$, 
$\mathrm{Fix}(f^n)$ is a compact subgroup of $\mathbb{T}^d$.
In the case when $\spec(M_f)$ does not contain $1$, 
the fixed points of all iterates are isolated and we have the following equality
\begin{equation}\label{eq:Fixfn=|det|}
\card\mathrm{Fix}(f^n)=\left|\det (\mathrm{id}-M_f^n)\right|
\end{equation}
This can be seen by counting $\mathbb{Z}^d$-lattice points contained 
in $M_f([0,1)^d)$, 
where $[0,1)^d$ is some fundamental domain in $\mathbb{R}^d$.
(When $\spec(M_f)$ contains $1$, 
$\mathrm{Fix}(f^n)$ is isomorphic to $\mathbb{T}^e$ 
for some positive integer $e\leq d$, 
and thus fixed points are never isolated.)
\begin{remark}
If $f$ possesses a Markov partition, and the corresponding transition matrix is given by $M$ then,
the equality below follows from the corresponding equality for the induced subshift of finite type $\sigma_f$:
\begin{equation}\label{eq:HypTorAut:Fix_pts-Trace}
\card \mathrm{Fix}(\sigma_f^n)=\trace(M^n)
\end{equation}
However, the encoding of orbits in this way can lead to a 
miscount of the number of periodic points that lie in the boundaries of the Markov rectangles.
\end{remark}
\noindent
Bowen~\cite[Corollary 16]{Bowen1971} showed that for a general toral automorphism $f$
\begin{equation}
h_\mathrm{top}(f)=\sum_{\lambda\in\mathrm{spec}(M_f):|\lambda|>1} \log |\lambda|
\end{equation}
(See also~\cite{Bowen1978b} or~\cite[Theorem 8.15]{WaltersBook}.)
When $d=2$, this reduces to the equality
\begin{equation}
h_\mathrm{top}(f)=\log \specrad(M_f)
\end{equation}
\begin{comment}
By a result of Bowen~\cite{Bowen1970b} 
the topological entropy of $f$ satisfies
\begin{equation}
h_{\mathrm{top}}(f)
=\limsup_{n\to\infty}\frac{1}{n}\log \card \mathrm{Fix}(f^n)
\end{equation}
(Bowen actually shows this equality holds more generally when
$M$ is an arbitrary compact manifold,
and $f\in\mathrm{Diff}^1(M)$ has (uniformly) hyperbolic non-wandering set.)
\end{comment}
When the toral automorphism $f$ is hyperbolic, and consequently
\footnote{Recall that a diffeomorphism is {\it Axiom A} if the nonwandering set is hyperbolic and the set of periodic points is 
dense in the non-wandering set.} 
Axiom A, another result of 
Bowen~\cite[Theorem 4.9]{Bowen1970b} implies that
\begin{equation}
h_\mathrm{top}(f)=\limsup_{n\to\infty}\frac{1}{n}\log \#\mathrm{Fix}(f^n)
\end{equation}
Consequently, if we denote by $\rho$ the radius of convergence of the Artin-Mazur dynamical zeta function $\zeta_f$ of 
the hyperbolic toral automorphism $f$, then
\begin{equation}
\rho
=\exp(-h_{\mathrm{top}}(f))
\end{equation}
%
%
\begin{comment}
\noindent
{\it Proof:\/}
We give a sketch in the case when $d=2$, 
$A=\left[\begin{array}{cc}1&1\\ 1&0\end{array}\right]$.
First, construct a Markov partition. 
This sets up a conjugacy (via itinerary) between $F$ and the subshift of finite type with transition matrix $A$. 
Let $\sigma$ denote the shift map for the subshift.
Here is a sketch of why the cardinality of $\mathrm{Fix}(\sigma^n)$ is the $n$th Fibonacci number $q_n$.

Recall, $\sigma$ the shift, restricted to 
\begin{equation}
\Sigma=\left\{(w_i)\in \{a,b\}^\mathbb{Z} : \ \mbox{word} \ bb \ \mbox{is forbidden} \right\}
\end{equation} 
({\it i.e.\/}, $\sigma$ is the subshift of finite type associated to the matrix $A$). 
Let $G(A)$ denote the directed graph with adjacency matrix $A$.

Fact 1: 
The number of solutions to $\sigma^n(x)=x$ equals the number of circuits of length $n$ that can be made in $G(A)$

Fact 2: 
The number of circuits of length $n$ that can be made on the graph $G(A)$ equals the number of circuits of length $1$ on the graph $G(A^n)$

Fact 3: 
For a general matrix $A'$ (whose entries are non-negative integers) 
the number of circuits of length $1$ of a graph $G(A')$ equals $\mathrm{tr}(A')$, the trace of $A'$

Fact 4: 
$A^n=\left[\begin{array}{cc}q_{n+1}&q_n\\q_n&q_{n-1}\end{array}\right]$ ,             [Hint: induction]
/\!/

\vspace{5pt}

\end{comment}
%
%
%
%
%
%
%
\section{Generating functions.}\label{sect:Gen_fn}
\subsection{Generating functions and the Gauss transformation.}\label{subsect:Gen_fn-Gauss_map}
Given $\theta\in[0,1]\setminus\mathbb{Q}$, denote the $n$th convergent of $\theta$ by $p_n/q_n$.
Given non-negative integers $r$ and $s$, consider the generating function
\begin{equation}\label{eq:F+G-defn}
F_{\pq{r}{s}}(z)=\sum_{n\geq 0}(p_n)^r(q_n)^s z^n 
\end{equation}
\begin{comment}
The radius of convergence of $F_{p^r q^s}$ is 
given (via the Cauchy-Hadamard Theorem) by
\begin{equation}
\rho_{F_{p^r q^s}}=\frac{1}{\limsup_{n\to\infty}|p_n|^{r/n}|q_n|^{s/n}} 
\end{equation}
\end{comment}
%
(Note that we will use parentheses when taking powers of $p_n$ and $q_n$ wherever possible, to avoid confusion with our $\caret$-notation introduced earlier.)
We wish to understand how $F_{\pq{r}{s}}$ behaves under the action of the Gauss transformation.
The following is a Corollary to Theorem~\ref{thm:p/q-action_under_Gauss_map}.
\begin{corollary}\label{cor:Gen_fn-action_under_Gauss_map}
For each $r\in\mathbb{N}$ the following holds.
Given 
$\theta\in [0,1]\setminus \mathbb{Q}$ 
let 
$\hat{\theta}=\mathrm{T}(\theta)$.
Let 
$p_n/q_n$ and $\hat{p}_n/\hat{q}_n$ 
denote the $n$th convergents for 
$\theta$ and $\hat\theta$ 
respectively. 
For each integer $s$ satisfying $0\leq s\leq r$, let
\begin{equation}\label{eq:F+hatF-defn}
F_{\pq{r-s}{s}}(z)=\sum_{n\geq 0}p_n^{r-s}q_n^s z^n \qquad
\hat{F}_{\pq{r-s}{s}}(z)=\sum_{n\geq 0}\hat{p}_n^{r-s}\hat{q}_n^{s}z^n
\end{equation}
Then
\begin{align}\label{eq:Gen_fn-action_under_Gauss_map} 
\left[\begin{array}{c}F_{\pq{r}{0}}(z)\\ F_{\pq{r-1}{1}}(z)\\ \vdots\\ F_{\pq{0}{r}}(z)\end{array}\right]
=
\left[\begin{array}{c}(p_0)^r(q_0)^0\\(p_0)^{r-1}(q_0)^1\\ \vdots\\ (p_0)^{0}(q_0)^{r}\end{array}\right]
+
z
E(a_1;r)
\left[\begin{array}{c}\hat{F}_{\pq{r}{0}}(z)\\ \hat{F}_{\pq{r-1}{1}}(z)\\ \vdots\\ \hat{F}_{\pq{0}{r}}(z)\end{array}\right]
\end{align}
\end{corollary}
%
\begin{comment}
\begin{proof}
Applying~\eqref{eq:F+hatF-defn} twice, 
together with Theorem~\ref{thm:p/q-action_under_Gauss_map}, gives
\begin{align}
&
\left[\begin{array}{c}F_{p^{r}q^{0}}(z)\\ F_{p^{r-1}q^{1}}(z)\\ \vdots\\ F_{p^{0}q^{r}}(z)\end{array}\right]
=
\left[\begin{array}{c}(p_0)^r(q_0)^0\\(p_0)^{r-1}(q_0)^1\\ \vdots\\ (p_0)^{0}(q_0)^{r}\end{array}\right]
+
\sum_{n\geq 1} z^n
\left[\begin{array}{c}(p_n)^r(q_n)^0\\(p_n)^{r-1}(q_n)^1\\ \vdots\\ (p_n)^{0}(q_n)^{r}\end{array}\right]\\
&=
\left[\begin{array}{c}(p_0)^r(q_0)^0\\(p_0)^{r-1}(q_0)^1\\ \vdots\\ (p_0)^{0}(q_0)^{r}\end{array}\right]
+
\sum_{n\geq 1} z^n
E(a_1;r)
\left[\begin{array}{c}(\hat{p}_{n-1})^r(\hat{q}_{n-1})^0\\(\hat{p}_{n-1})^{r-1}(\hat{q}_{n-1})^1\\ \vdots\\ (\hat{p}_{n-1})^{0}(\hat{q}_{n-1})^{r}\end{array}\right]\\
&=
\left[\begin{array}{c}(p_0)^r(q_0)^0\\(p_0)^{r-1}(q_0)^1\\ \vdots\\ (p_0)^{0}(q_0)^{r}\end{array}\right]
+zE(a_1;r)
\left[\begin{array}{c}\hat{F}_{p^{r}q^{0}}(z)\\ \hat{F}_{p^{r-1}q^{1}}(z)\\ \vdots\\ \hat{F}_{p^{0}q^{r}}(z)\end{array}\right]
\end{align}
\end{proof}
\end{comment}
%
We now consider iterating the action of the Gauss transformation.
Let us adopt the notation from Section~\ref{subsubsect:Gauss_Map}.
For each non-negative integer $n$, let 
$\theta^{\caret n}=\mathrm{T}^n(\theta)$ 
and take 
$a^{\caret n}_m$, 
$p^{\caret n}_m$, $q^{\caret n}_m$, etc., as before.
For each positive integer $m$ define
\begin{equation}
F_{\pq{r}{s}}^{\caret m}(z)=\sum_{n\geq 0}(p_n^{\caret m})^r(q_n^{\caret m})^s z^n 
\end{equation}
Applying 
Corollary~\ref{cor:Gen_fn-action_under_Gauss_map} iteratively, 
together with Corollary~\ref{cor:pn/qn-action_under_Gauss_map}, 
we get the following.
\begin{corollary}\label{cor:Gen_fn-action_under_Gauss_map_2}
For each $r\in\mathbb{N}$ the following holds.
Given $\theta\in[0,1]\setminus \mathbb{Q}$ let $\theta^{\caret m}=\mathrm{T}^m(\theta)$ for each positive integer $m$. 
For each integer $s$ satisfying $0\leq s\leq r$, let
$F_{\pq{r-s}{s}}$ and $F_{\pq{r-s}{s}}^{\caret m}$ 
denote the generating functions defined above corresponding to $\theta$ and $\theta^{\caret m}$ respectively.
Then, for each positive integer $m$,
\begin{align*}
\left[\begin{array}{c}F_{\pq{r}{0}}(z)\\ F_{\pq{r-1}{1}}(z)\\ \vdots\\ F_{\pq{0}{r}}(z)\end{array}\right]
&=
\!
\sum_{0\leq n<m}z^n\left[\begin{array}{c}(p_n)^r(q_n)^0\\ (p_n)^{r-1}(q_n)^1\\ \vdots\\ (p_n)^0(q_n)^r\end{array}\right]
+
z^m E
\left[\begin{array}{c}F^{\caret m}_{\pq{r}{0}}(z)\\ F^{\caret m}_{\pq{r-1}{1}}(z)\\ \vdots\\ F^{\caret m}_{\pq{0}{r}}(z)\end{array}\right]
\end{align*}
where 
$E=E(a_1;r)E(a_2;r)\cdots E(a_m;r)$.
\end{corollary}
\begin{remark}\label{rmk:Gen_fn-action_under_Gauss_map_2}
In particular, for $r=1$ if we adopt the alternative notation 
$F_p=F_{\pq{1}{0}}$ and $F_q=F_{\pq{0}{1}}$ then we get the following
\begin{align*}
\left[\begin{array}{l}F_{p}(z)\\ F_{q}(z)\end{array}\right]
=
\!
\sum_{0\leq n<m}z^n\left[\begin{array}{l}p_n\\ q_n\end{array}\right]
+
z^m
\left[\begin{array}{ll}0&1\\ 1&a_1\end{array}\right]
\left[\begin{array}{ll}0&1\\ 1&a_2\end{array}\right]
\cdots
\left[\begin{array}{ll}0&1\\ 1&a_m\end{array}\right]
\left[\begin{array}{l}F^{\caret m}_{p}(z)\\ F^{\caret m}_{q}(z)\end{array}\right]
\end{align*}
\end{remark}
%
%
%
%
%
\subsection{Generating functions associated to quadratic irrationals.}\label{subsect:gen_fn-quad_irrational}
We now focus on the case when 
$\theta\in[0,1]\setminus\mathbb{Q}$ 
is a quadratic irrational.
Thus, for some positive integers $a_1,a_2,\ldots,a_{\kay+\ell}\in\mathbb{N}$,
\begin{equation}\label{eq:ctd_frac-eventually_periodic}
\theta=\left[a_1,a_2,\ldots,a_\kay,\overline{a_{\kay+1},\ldots,a_{\kay+\ell}}\right]
\end{equation}
As above, let $p_n/q_n$ denote the sequence of convergents of $\theta$ 
and 
take the corresponding generating functions $F_{\pq{r}{s}}$ defined by~\eqref{eq:F+G-defn}, as before.
The aim of this section is to prove the following generalisation of Theorem~\ref{thm:gen_fn_rat}.
\begin{theorem}\label{thm:gen_fn_rat->higher-powers}
For each $r\in\mathbb{N}$ the following holds.
Let 
$\theta\in[0,1]\setminus\mathbb{Q}$ 
be a quadratic irrational with continued fraction expansion
\begin{equation}
\theta=[a_1,a_2,\ldots,a_\kay,\overline{a_{\kay+1},\ldots,a_{\kay+\ell}}]
\end{equation}
For each non-negative integer $n$, let $p_n/q_n$ denote the $n$th convergent of $\theta$.
Then the associated generating functions $F_{\pq{r-s}{s}}(z)$ are rational functions in the variable $z$, with integer coefficients.
In fact, if
\begin{equation}\label{eq:N_0+N_1-r_arb}
N_0=E(a_1;r)\cdots E(a_\kay;r) \qquad
N_1=E(a_{\kay+1};r)\cdots E(a_{\kay+\ell};r)
\end{equation}
and
\begin{equation}\label{eq:P_0+P_1-r_arb}
P_0(z)=\!
\sum_{0\leq n<\kay}z^n
\left[\begin{array}{c}
(p_n)^r(q_n)^0\\
(p_n)^{r-1}(q_n)^1\\
\vdots\\
(p_n)^0(q_n)^r
\end{array}\right] \quad
P_1(z)=\!
\sum_{\kay\leq n<\kay+\ell}z^{n}\left[\begin{array}{c}
(p_n)^r(q_n)^0\\
(p_n)^{r-1}(q_n)^1\\
\vdots\\
(p_n)^0(q_n)^r
\end{array}\right]
\end{equation}
then
\begin{align}\label{eq:gen_fn_rat-r-arb}
\left[\begin{array}{c}
F_{\pq{r}{0}}(z)\\
F_{\pq{r-1}{1}}(z)\\
\vdots\\
F_{\pq{0}{r}}(z)
\end{array}\right]
=
P_0(z)
+\left(\mathrm{id}-z^\ell N_0N_1N_0^{-1}\right)^{-1}P_1(z)
\end{align}
\end{theorem}
\begin{proof}
Let $\theta$ have continued fraction expansion given by expression~\eqref{eq:ctd_frac-eventually_periodic} above.
Then $\theta^{\caret \kay}=\theta^{\caret \kay+\ell}$. 
Therefore 
$p^{\caret \kay}_n=p^{\caret \kay+\ell}_n$ and 
$q^{\caret \kay}_n=q^{\caret \kay+\ell}_n$, for all $n$.
Consequently, for each integer $s$ satisfying $0\leq s\leq r$,
\begin{equation}\label{eq:F^m+l_r-s,s=F^m_r-s,s}
F^{\caret \kay}_{\pq{r-s}{s}}(z)=F^{\caret \kay+\ell}_{\pq{r-s}{s}}(z)
\end{equation}
By Corollary~\ref{cor:Gen_fn-action_under_Gauss_map_2} we therefore get
\begin{align}\label{eq:F^m+l=F^m}
&
\left[\begin{array}{c}
F_{\pq{r}{0}}(z)\\
F_{\pq{r-1}{1}}(z)\\
\vdots\\
F_{\pq{0}{r}}(z)
\end{array}\right]
=
\sum_{0\leq n<\kay}z^n\left[\begin{array}{c}
(p_n)^r(q_n)^0\\
(p_n)^{r-1}(q_n)^1\\
\vdots\\
(p_n)^0(q_n)^r
\end{array}\right]
+z^\kay N_0
\left[\begin{array}{c}
F^{\caret \kay}_{\pq{r}{0}}(z)\\
F^{\caret \kay}_{\pq{r-1}{1}}(z)\\
\vdots\\
F^{\caret \kay}_{\pq{0}{r}}(z)
\end{array}\right]\\
&\qquad\quad \; =
\sum_{0\leq n<\kay+\ell}z^n\left[\begin{array}{c}(p_n)^r(q_n)^0\\ (p_n)^{r-1}(q_n)^1\\ \vdots\\ (p_n)^0(q_n)^r\end{array}\right]
+
z^{\kay+\ell} N_0 N_1
\left[\begin{array}{c}
F^{\caret \kay+\ell}_{\pq{r}{0}}(z)\\
F^{\caret \kay+\ell}_{\pq{r-1}{1}}(z)\\
\vdots\\
F^{\caret \kay+\ell}_{\pq{0}{r}}(z)
\end{array}\right]
\end{align}
where $N_0$ and $N_1$ are given by equations~\eqref{eq:N_0+N_1-r_arb}.
Upon rearranging the above equality and using 
the equality~\eqref{eq:F^m+l_r-s,s=F^m_r-s,s} above
we therefore find that
\begin{equation}
z^\kay N_0\left(\mathrm{id}-z^\ell N_1\right)
\left[\begin{array}{c}
F_{\pq{r}{0}}^{\caret\kay}(z)\\
F_{\pq{r-1}{1}}^{\caret\kay}(z)\\
\vdots\\
F_{\pq{0}{r}}^{\caret\kay}(z)
\end{array}\right]
=
\sum_{\kay\leq n<\kay+\ell}z^n\left[\begin{array}{c}
(p_n)^r(q_n)^0\\
(p_n)^{r-1}(q_n)^1\\
\vdots\\
(p_n)^0(q_n)^r
\end{array}\right]
\end{equation}
Since $N_0$ is nonsingular, 
the matrix $z^\kay N_0(\mathrm{id}-z^\ell N_1)$ 
is invertible provided that 
$z^\ell$ does not lie in $\spec (N_1)\cup \{0\}$.
Provided that this is the case
\begin{equation}
\left[\begin{array}{c}
F_{\pq{r}{0}}^{\caret\kay}(z)\\
F_{\pq{r-1}{1}}^{\caret\kay}(z)\\
\vdots\\
F_{\pq{0}{r}}^{\caret\kay}(z)
\end{array}\right]
=
\left(\mathrm{id}-z^\ell N_1\right)^{-1}N_0^{-1}
\sum_{\kay\leq n<\kay+\ell}z^{n-\kay}\left[\begin{array}{c}(p_n)^r(q_n)^0\\ (p_n)^{r-1}(q_n)^1\\ \vdots\\ (p_n)^0(q_n)^r\end{array}\right]
\end{equation}
Substituting into equation~\eqref{eq:F^m+l=F^m} 
gives the equation~\eqref{eq:gen_fn_rat-r-arb} 
with $P_0(z)$ and $P_1(z)$ given by equations~\eqref{eq:P_0+P_1-r_arb}, 
as required.
\end{proof}
\begin{remark}\label{rmk:P1_in_terms_of_A}
From the definition in Theorem~\ref{thm:gen_fn_rat->higher-powers}
above, together with Corollary~\ref{cor:pn/qn-action_under_Gauss_map},
we find that
\begin{align}
P_1(z)
=
\sum_{\kay\leq n<\kay+\ell}z^{n}
E(a_1;r)E(a_2;r)\cdots E(a_n;r)
e_{r+1}
\end{align}
where
$e_{r+1}
=\left[\begin{array}{c}
(p_0)^r(q_0)^0\\
\vdots\\
(p_0)^1(q_0)^{r-1}\\
(p_0)^0(q_0)^r
\end{array}\right]
=
\left[\begin{array}{c}
0\\
\vdots\\
0\\
1
\end{array}\right]$.
The matrix $E(a;r)$ is non-negative and,
when $a>0$, $E(a;r)$ maps $e_{r+1}$ to a positive vector.
As none of the rows of $E(a;r)$ vanish, positive vectors are also mapped to positive vectors. 
Thus $P_1(z)$ is a real-linear combination of non-negative vectors and, for $\ell\neq 1$, at least one will be positive.
\end{remark}
\begin{corollary}\label{cor:gen_fn_rat-radius}
Given a quadratic irrational $\theta$, let $N_1$
denote the matrix corresponding to the periodic part of the continued fraction expansion of $\theta$.
For each $r\in\mathbb{N}$, 
the generating functions $F_{r-s,s}(z)$, $0\leq s\leq r$,  
has radius of convergence
\begin{equation}
\rho=[\specrad(N_1)]^{-1/\ell}
\end{equation}
\end{corollary}
\begin{proof}
For notational simplicity let $N=N_0 N_1 N_0^{-1}$.
The strategy is to show that
\begin{equation}\label{eq:pole_at_rho}
\lim_{z\to \rho^-} (\id-z^\ell N)^{-1}P_1(z)=\infty
\end{equation}
Since 
$(\id-z^\ell N)^{-1} P_1(z)$ 
necessarily has no poles of absolute value less than $\rho$ (as $P_1(z)$ is polynomial) the result will follow.
Below we will prove the limit~\eqref{eq:pole_at_rho} in the periodic case, {\it i.e.}, when $\kay=0$.
The general case will follow as, by Remark~\ref{rmk:P1_in_terms_of_A},
\begin{align}
&
(\id-z^\ell N)^{-1}P_1(z)\notag\\
&=
N_0(\id-z^\ell N_1)^{-1}N_0^{-1}z^{\kay}\sum_{\kay\leq n\leq \kay+\ell}N_0 E(a_{\kay+1};r)\cdots E(a_{\kay+n};r)e_{r+1}\\
&=
z^\kay N_0\cdot (\id-z^\ell N_1)^{-1}\sum_{\kay\leq n\leq \kay+\ell}E(a_{\kay+1};r)\cdots E(a_{\kay+n};r)e_{r+1}
\end{align}
Henceforth we only consider the periodic case.
Observe that the matrix $N_1$ is non-negative and primitive 
\footnote{Recall that a non-negative matrix $M$ is {\it primitive} 
if $M^n$ is positive for some $n\in\mathbb{N}$. In this particular case $N_1^2$ is positive.}.
Thus, by the Perron-Frobenius theorem, 
$\lambda=\specrad(N_1)$ 
is an eigenvalue of $N_1$ and all other eigenvalues have strictly smaller modulus.
This Perron-Frobenius eigenvalue is simple with one-dimensional eigenspace which is the span of some positive right eigenvector $v$.
There is also a corresponding left eigenvector $w$ of $N_1$ which is also positive and, if $P_v$ denotes the projection onto 
$\mathrm{span}(v)$ then $v\cdot w^\top=P_v$, the projection onto $\mathrm{span}(v)$. 
Moreover, for all positive integers $n$,
\begin{equation}\label{ineq:PF_rate}
\|N_1^n/\lambda^n-P_v\|\leq C\lambda_2^n
\end{equation}
where $\lambda_2$ denotes the absolute value of the next largest eigenvalue and $C$ is some positive constant independent of $n$.
See, for instance~\cite[Section 4.5]{LindMarcusBook}. 
(Note -- the rate of convergence is not stated, but follows by observing that if the matrix $M$ has 
eigenvalues all of modulus less than one, then $\|M^n\|\leq r^n$ for all positive $n$, where $r=\max_{s\in\spec M}|s|$.)

For $|z|<\lambda^{-1/\ell}$, the (modified) Neumann series
\begin{equation}
(\id-z^\ell N)^{-1}=\id+z^\ell N+z^{2\ell}N^2+\cdots
\end{equation}
converges absolutely.
Take $z\in (\lambda_2^{-1/\ell},\lambda^{-1/\ell})$.
Since $z$ is positive, for any $n\in\mathbb{N}$ we have $E(a_1;r)\cdots E(a_n;r)e_{r+1}\geq 0$ and hence $P_1(z)\geq 0$.
Then $u=N_1^2 P_1(z)>0$.
Moreover, the Neumann series applied to $P_1(z)$ is positive:
\begin{equation}
0\leq P_1(z)+z^\ell N_1P_1(z)+z^{2\ell}
\sum_{n\geq 0}z^{n\ell}N_1^n u(z)
\end{equation}
Set $\mu=z^{-\ell}$. 
By inequality~\eqref{ineq:PF_rate}, the sum on the right-hand side is
\begin{align}
\sum_{n\geq 0}\mu^{-n}N_1^n u(\mu)
&=
\sum_{n\geq 0}\mu^{-n}\left[\lambda^n P_v u(\mu)+O(\lambda_2^n)\right]
\end{align}
Taking the limit $z\to \frac{1}{\lambda^{1/\ell}}^-$ is equivalent taking the limit to $\mu\to \lambda^+$. 
But observe that 
$\lim_{\mu\to \lambda^+}\left[\sum_{n\geq 0} (\lambda/\mu)^n\right] P_v u(\mu)$ 
diverges to infinity and 
$\lim_{\mu\to \lambda^+}\sum_{n\geq 0} (\lambda_2/\mu)^n$ is convergent.
Therefore, since the Perron-Frobenius eigenvector $v$ is positive, each entry of the vector 
$\left[\sum_{n\geq 0} (\lambda/\mu)^n\right] P_v u(\mu)$ 
must diverge to infinity as $\mu\to\lambda^+$.
Thus the result follows. 
\end{proof}
\begin{corollary}\label{cor:Levy_constant_formula} 
Given a quadratic irrational $\theta$, let $N_1$
denote the matrix corresponding to the periodic part of the continued fraction expansion of $\theta$.
Then the L\'evy constant $\beta(\theta)$ of $\theta$ exists and
\begin{equation}
\beta(\theta)=\frac{1}{\ell} \log\specrad (N_1)
\end{equation}
\end{corollary}
\begin{proof}
For a quadratic irrational $\theta$, the existence of the L\'evy constant $\beta(\theta)$ is relatively straightforward.
(See, for instance,~\cite{JagerLiardet1988} or Appendix~\ref{sect:Gauss_Map}).
For arbitrary irrationals $\theta$ (not just quadratic irrationals), when $\beta(\theta)$ exists we know that 
\begin{equation}
\beta(\theta)=\log\left(\lim_{n\to\infty}|q_n|^{1/n}\right)=-\log \rho
\end{equation}
where $\rho$ denotes the radius of convergence of $F_{0,1}$ (and hence also of $F_{1,0}$).
Therefore, by Corollary~\ref{cor:gen_fn_rat-radius} above the result follows.
\end{proof}
%
%
%
%
%

\section{Prime geodesics on the modular surface and elements of the mapping class group of the torus.}
%
%
\subsection{The Mapping Class Group of the Torus.}\label{subsect:MCL_T2}
We collect here some basic facts about the moduli space of the $2$-torus $\mathbb{T}^2$.
The reader should consult~\cite{FarbMargalitBook2012}
for further details.

Let
$\mathrm{Mod}(\mathbb{T}^2)$ 
denote the space of complex structures on $\mathbb{T}^2$.
Equivalently, 
$\mathrm{Mod}(\mathbb{T}^2)$ 
is the space of complex curves of genus $1$.
Each complex curve of genus $1$ is biholomorphic to $\mathbb{C}/\Lambda$, for some (non-degenerate) oriented
\footnote{Here we do not assume that the orientation necessarily agrees with the standard orientation on $\mathbb{C}$ 
-- thus lattices of both orientations are considered.} 
lattice 
$\Lambda\subset \mathbb{C}$.
We say that the lattices 
$\Lambda$ and $\Lambda'$ in $\mathbb{C}$ 
are equivalent, written 
$\Lambda\sim\Lambda'$, if the quotient spaces 
$\mathbb{C}/\Lambda$ and $\mathbb{C}/\Lambda'$
are biholomorphically equivalent.
Consider the lattices 
$\Lambda=\omega_1\mathbb{Z}\oplus\omega_2\mathbb{Z}$ 
and
$\Lambda'=\omega_1'\mathbb{Z}\oplus\omega_2'\mathbb{Z}$.
Then
\begin{enumerate}
\item
$\Lambda=\Lambda'$ if and only if there exist $a,b,c,d\in\mathbb{Z}$ such that $ad-bc=1$ and
\begin{equation*}
\omega_1'=a\omega_1+b\omega_2 \qquad 
\omega_1'=c\omega_1+d\omega_2
\end{equation*}
\item
$\Lambda\sim\Lambda'$ if and only if $\Lambda'=\lambda\cdot \Lambda$, for some 
$\lambda\in\mathbb{C}\setminus\{0\}$
\end{enumerate}
%
\begin{comment}
\noindent
{\it Proof:\/}
\begin{enumerate}
\item
The condition is just a statement that there is an invertible change of basis from 
$\{\omega_1,\omega_2\}$ 
to 
$\{\omega_1',\omega_2'\}$
and the condition $ad-bc=1$ implies the orientation of the bases is preserved and is restricts to a bijection from $\Lambda$ to $\Lambda'$.
\item
This follows trivially as the linear map $z\mapsto \lambda z$ is a biholomorphism. 
\end{enumerate}
/\!/

\vspace{5pt}
\end{comment}
By property (2) it follows that the lattice 
$\Lambda=\omega_1\mathbb{Z}\oplus\omega_2\mathbb{Z}$ 
is equivalent to the lattice
$\Lambda=\mathbb{Z}\oplus\tau\mathbb{Z}$, 
where 
$\tau=\omega_2/\omega_1$ is non-real.
Therefore, by property (1) 
we can identify
$\mathrm{Mod}(\mathbb{T}^2)$ 
with the space
\begin{equation}
\left\{\mathbb{Z}\oplus\tau\mathbb{Z} : \tau\in \mathbb{H}^-\cup\mathbb{H}^+\right\}/\mathrm{PSL}(2,\mathbb{Z})
\end{equation}
Hence, the double of the modular surface
$(\mathbb{H}^-\cup\mathbb{H}^+)/\mathrm{PSL}(2,\mathbb{Z})$ 
can be identified with 
$\mathrm{Mod}(\mathbb{T}^2)$ 
via the mapping
$\tau\mapsto \mathbb{Z}\oplus\tau\mathbb{Z}$.

Fix an oriented closed curve $\gamma$ in 
$
\mathrm{Mod}(\mathbb{T}^2)$, 
together with a point $\tau\in\gamma$.
Consider an arbitrary lift $\tilde\gamma$ of the geodesic $\gamma$ to the universal cover 
$
\widetilde{\mathrm{Mod}}(\mathbb{T}^2)$ 
and let $\tilde\tau$ denote the corresponding lift of $\tau$.
Abusing notation slightly, we denote the geodesic 
completion of the lift of $\gamma$ also by $\tilde\gamma$. 
Let $M_t\in\mathrm{PSL}(2,\mathbb{R})$, $t\in\mathbb{R}$, 
denote the one-parameter hyperbolic subgroup leaving $\tilde\gamma$ 
invariant and fixing the ends of $\tilde\gamma$.  
Assume that $M_t$ is parametrised so that $M=M_1$ is prime
and such that for any point on $\tilde\gamma$ the orientation of 
the geodesic segment from the point to its image under $M$ agrees with the 
orientation of $\tilde\gamma$.
Let 
$\Lambda_0$
denote the lattice corresponding to $\tilde\tau$, {\it i.e.}, in vector notation 
\begin{equation}
\Lambda_{0}=
\left[\begin{array}{c}1\\0\end{array}\right]\mathbb{Z}\oplus
\left[\begin{array}{c}\Re\tilde\tau\\ \Im\tilde\tau\end{array}\right]\mathbb{Z}
\end{equation}
and, for $t\in\mathbb{R}$, define $\Lambda_t=M_t(\Lambda_0)$, {\it i.e.}, in vector notation the lattice given by
\begin{equation}
\Lambda_t=
M_t\left[\begin{array}{c}1\\0\end{array}\right]\mathbb{Z}\oplus
M_t\left[\begin{array}{c}\Re\tilde\tau\\ \Im\tilde\tau\end{array}\right]\mathbb{Z}
\end{equation}
Since 
$M=M_1\in\mathrm{PSL}(2,\mathbb{Z})$, 
$M$ preserves the lattice $\Lambda_0$, 
{\it i.e.}, 
$\Lambda_0$ is equivalent to $\Lambda_1$.
Thus $M$ descends to a holomorphic self-map of 
$\mathbb{C}/\Lambda_0$.
In fact, $M$ induces a hyperbolic toral automorphism of 
$\mathbb{C}/\Lambda_0$.
Moreover, if any other point $\tau'\in\gamma$ is chosen, 
the resulting hyperbolic toral automorphism  
will be conjugate 
(where the conjugacy is induced by an element of $\mathrm{PGL}(2,\mathbb{R})$) 
to the induced hyperbolic toral automorphism corresponding to $\tau$.
Consequently, by normalising to the standard 2-torus 
$\mathbb{T}^2\simeq\mathbb{C}/(\mathbb{Z}+\imath\mathbb{Z})$ 
or otherwise,
the induced toral automorphism is independent of the point $\tau$.
Similarly, for any other lift of $\gamma$ the corresponding 
hyperbolic toral automorphism constructed above will be conjugate 
(where, this time, the conjugacy is induced by an element of $\mathrm{PGL}(2,\mathbb{Z})$).    
%
\begin{comment}
\noindent
{\it Proof:\/}
The only non-trivial part is showing that $M$ induces a hyperbolic toral automorphism.
But this follows as $\gamma$ being a closed geodesic, implies that $M$ fixes the distinct endpoints $\gamma_-$ and $\gamma_+$ of the lift $\tilde\gamma$.
As $M$ has exactly two fixed points, this means that $\mathrm{tr}(M)^2-4>0$, and hence the matrix $M$ has two distinct real eigenvalues. 
Therefore $M$ is hyperbolic matrix, and thus descends to a hyperbolic toral automorphism.
/\!/

\vspace{5pt}

\end{comment}

\subsection{Mapping class elements associated to quadratic irrationals.}
Let $\theta\in[0,1]\setminus\mathbb{Q}$ be a quadratic irrational and let $\theta'$ denote the Galois conjugate.
Denote the corresponding pre-periodic simple continued fraction expansion of $\theta$ by
\begin{equation}\label{eq:simple_ctd_frac_exp-preperiodic}
[a_1,a_2,\ldots,a_{\kay},\overline{a_{\kay+1},a_{\kay+2},\ldots,a_{\kay+\ell}}]
\end{equation}
In Section~\ref{subsubsect:Quad_Irr} it was stated that the matrix 
$N=N_\theta=N_0^{-1}N_1N_0$ 
is hyperbolic, and is prime when $\ell$ is even and either
(a) is the minimal period of the sequence $a_{\kay+1},\ldots,a_{\kay+\ell}$, or
(b) is twice the minimal period of the sequence $a_{\kay+1},\ldots,a_{\kay+\ell}$ which is odd.
In this case $N_1$ corresponds to a conjugate-related pair of geodesics 
$\gamma^-$ and $\gamma^+$ in $\mathbb{H}^-\cup\mathbb{H}^+$ with ends $\theta$ and $\theta'$.
In the preceding Section~\ref{subsect:MCL_T2} we saw that
the geodesic $\gamma^\pm$ induces a closed path on $\mathrm{Mod}(\mathbb{T}^2)$ which,
in turn, induces a hyperbolic toral automorphism on $\mathbb{T}^2$.
In fact, this is the hyperbolic toral automorphism $f_N\colon \mathbb{T}^2\to\mathbb{T}^2$ induced by the matrix $N$.
\begin{remark}
The construction in Section~\ref{subsect:MCL_T2} also works 
in case (b) when the minimal period is odd 
(so that $N_1$ is not prime but $N_1^2$ is).
The hyperbolic element $N_1$ interchanges the geodesics $\gamma^-$ and $\gamma^+$.
These geodesics induce a pair of one-parameter families of lattices $\Lambda_t^-$ and $\Lambda_t^+$, $t\in\mathbb{R}$, 
which are conjugate-related,
{\it i.e.}, $\Lambda_t^\pm=\overline{\Lambda_t^\mp}$. 
This results in a pair of orientation-reversing linear torus maps 
$f_{N,\pm}\colon\mathbb{C}/\Lambda_1^\pm\to\mathbb{C}/\Lambda_1^\mp$. 
where $f_{N,-}\circ f_{N,+}=f_{N,+}\circ f_{N,-}=f_N$.
\end{remark}
We relate this to the L\'evy constant of a quadratic irrational in the following way.
\begin{proof}[Proof of Theorem~\ref{thm:Levy_vs_entropy}]
Take $\theta$ to be the quadratic irrational with
simple continued fraction expansion given by~\eqref{eq:simple_ctd_frac_exp-preperiodic}.
It was shown in Corollary~\ref{cor:Levy_constant_formula} 
(see also~\cite{JagerLiardet1988,BelovaHazard2017a})
that
the L\'{e}vy constant is given by
\begin{equation}
\beta(\theta)=\frac{1}{\ell}\log\specrad(N_1)
\end{equation}
where
\begin{equation}
N_1=
\left[\begin{array}{cc}0&1\\ 1& a_{\kay+1}\end{array}\right]
\left[\begin{array}{cc}0&1\\ 1& a_{\kay+2}\end{array}\right]
\cdots
\left[\begin{array}{cc}0&1\\ 1& a_{\kay+\ell}\end{array}\right]
\end{equation}
Since $N_1$ and $N$ are conjugate, it follows that $\specrad(N_1)=\specrad(N)$.
As was demonstrated in the previous section, $N$ 
is the hyperbolic linear transformation corresponding to the hyperbolic toral automorphism associated with $\theta$.
Consequently we get a first proof of Theorem~\ref{thm:Levy_vs_entropy}, as required.
\end{proof}
We will now consider relation between $f_\theta=f_{N_\theta}$ 
and the continued fraction expansion of $\theta$ in more detail by applying the results of Section~\ref{sect:Gen_fn} in the case $r=1$.
For simplicity of notation, let us write $F_p$ and $F_q$ 
for $F_{\pq{1}{0}}$ and $F_{\pq{0}{1}}$ respectively. 
We will also write $E(a)$ for $E(a;1)$.
The following is a straightforward calculation, which we leave to the reader.
\begin{proposition}\label{prop:mx_gen_fn-shift}
For each non-negative integer $m$,
\begin{equation}
\sum_{n\geq 0}z^n
\left[\begin{array}{cc}p_n^{\caret m}&p_{n+1}^{\caret m}\\ q_n^{\caret m}&q_{n+1}^{\caret m}\end{array}\right]
=
\left[\begin{array}{cc}F_p^{\caret m}(z)&z^{-1}(F_p^{\caret m}(z)-p_0^{\caret m})\\ F_q^{\caret m}(z)&z^{-1}(F_q^{\caret m}(z)-q_0^{\caret m})\end{array}\right]
\end{equation}
\end{proposition}
\begin{proposition}\label{prop:mx_gen_fn-action_of_Gauss_map}
For each positive integer $m$,
if
$N_0=E(a_1)\cdots E(a_m)$
then
\begin{align}
&
\sum_{n\geq 0} z^n
\left[\begin{array}{cc}p_n&p_{n+1}\\ q_n&q_{n+1}\end{array}\right]
\notag\\
&=
\sum_{0\leq n<m} z^n
\left[\begin{array}{cc}p_n&p_{n+1}\\ q_n&q_{n+1}\end{array}\right]
+z^m N_0
\left[\begin{array}{cc}F_p^{\caret m}(z)&z^{-1}(F_p^{\caret m}(z)-p_0^{\caret m})\\ F_q^{\caret m}(z)&z^{-1}(F_q^{\caret m}(z)-q_0^{\caret m})\end{array}\right]
\end{align}
\end{proposition}
\begin{proof}
By the preceding Proposition~\ref{prop:mx_gen_fn-shift} in the case $m=0$, 
together with Corollary~\ref{cor:Gen_fn-action_under_Gauss_map_2}
we find that
\begin{align}
&
\sum_{n\geq 0} z^n
\left[\begin{array}{c}p_n\\ q_n\end{array}\right]
=
\left[\begin{array}{c}F_p(z)\\ F_q(z)\end{array}\right]\\
&=
\sum_{0\leq n<m} z^n
\left[\begin{array}{c}p_n\\ q_n\end{array}\right]
+z^m E(a_1)\cdots E(a_m)
\left[\begin{array}{c}F_p^{\caret m}(z)\\ F_q^{\caret m}(z)\end{array}\right]\\
&
\sum_{n\geq 0} z^n
\left[\begin{array}{c}p_{n+1}\\ q_{n+1}\end{array}\right]
=
\left[\begin{array}{c}z^{-1}(F_p(z)-p_0)\\ z^{-1}(F_q(z)-q_0)\end{array}\right]\\
&=
\sum_{0\leq n<m} z^{n}
\left[\begin{array}{c}p_{n+1}\\ q_{n+1}\end{array}\right]
+z^m E(a_1)\cdots E(a_m)
\left[\begin{array}{c}z^{-1}(F_p^{\caret m}(z)-p_0^{\caret m})\\ z^{-1}(F_q^{\caret m}(z)-q_0^{\caret m})\end{array}\right]
\end{align}
Here we have used that 
$\left[\begin{array}{c}p_m\\ q_m\end{array}\right]
=
E(a_1)\cdots E(a_m)
\left[\begin{array}{c}p_0^{\caret m}\\ q_0^{\caret m}\end{array}\right]$ 
(see Corollary~\ref{cor:pn/qn-action_under_Gauss_map} and the subsequent Remark~\ref{cor:pn/qn-action_under_Gauss_map}), 
then cancelled $z^{-1}$ and re-indexed the sum.
Setting 
$N_0=E(a_1)\cdots E(a_m)$ 
and adjoining the column vectors 
now gives the result.
\end{proof}
\begin{proposition}\label{prop:mx_gen_fn-periodic}
If we define
\begin{equation}
N_1
=E(a_1^{\caret\kay})\cdots E(a_\ell^{\caret\kay})
=E(a_{\kay+1})\cdots E(a_{\kay+\ell})
\end{equation}
and set
\begin{equation}
W(z)=
\sum_{n\geq 0} z^{n\ell} N_1^n
\qquad
V(z)=
\sum_{0\leq k<\ell}z^k E(a_{\kay+1})\cdots E(a_{\kay+k+1})
\end{equation}
then the following equality holds,
\begin{equation}
\sum_{n\geq 0}z^n
\left[\begin{array}{cc}p_n^{\caret \kay}&p_{n+1}^{\caret \kay}\\ q_n^{\caret \kay}&q_{n+1}^{\caret \kay}\end{array}\right]
=
W(z)V(z)
\end{equation}
\end{proposition}
\begin{remark}\label{rmk:N1_in_terms_of_pq} 
Recall that, for each non-negative integer $n$,
\begin{equation}
\left[\begin{array}{cc}p_{n}^{\caret\kay}& p_{n+1}^{\caret\kay}\\ q_{n}^{\caret\kay}& q_{n+1}^{\caret\kay}\end{array}\right]
=E(a_1^{\caret\kay})\cdots E(a_{n+1}^{\caret\kay})
\end{equation}
Thus the left-hand side above coincides with the matrix $N_1$ in Proposition~\ref{prop:mx_gen_fn-periodic}.
\end{remark}
\begin{proof}[Proof of Proposition~\ref{prop:mx_gen_fn-periodic}]
By Remark~\ref{rmk:N1_in_terms_of_pq} above, 
and since $a_{n+\ell}=a_{n}$ for all $n\geq\kay+1$, 
we find that
\begin{align}
&
\sum_{n\geq 0}z^n
\left[\begin{array}{cc}p_n^{\caret \kay}&p_{n+1}^{\caret \kay}\\ q_n^{\caret \kay}&q_{n+1}^{\caret \kay}\end{array}\right]\notag\\
&=
\sum_{n\geq 0}
z^n E(a_{\kay+1})\cdots E(a_{\kay+n+1})\\
&=
\sum_{n\geq 0}
\sum_{0\leq m<\ell}
z^{n\ell} \left(E(a_{\kay+1})\cdots E(a_{\kay+\ell})\right)^n 
\cdot z^m E(a_{\kay+1})\cdots E(a_{\kay+m+1})\\
&=
\left(\sum_{n\geq 0} z^{n\ell} N_1^n\right)
\left(\sum_{0\leq m<\ell}z^m E(a_{\kay+1})\cdots E(a_{\kay+m+1})\right)
\end{align}
The result now follows.
\end{proof}
Rearranging the equation in Proposition~\ref{prop:mx_gen_fn-action_of_Gauss_map} when $m=\kay$,
and applying Proposition~\ref{prop:mx_gen_fn-shift} and Proposition~\ref{prop:mx_gen_fn-periodic} above, together with 
Theorem~\ref{thm:gen_fn_rat->higher-powers},
now gives the following.
\begin{corollary}\label{cor:W-rational}
Let $W(z)$ be given by Proposition~\ref{prop:mx_gen_fn-periodic} above.
Then
\begin{equation}
W(z)=U(z)^{-1}(X(z)-Y(z))V(z)^{-1}
\end{equation}
where
\begin{align}
U(z)
&=z^\kay E(a_1)E(a_2)\cdots E(a_\kay)\\
V(z)
&=\sum_{0\leq m<\ell}z^m E(a_{\kay+1})\cdots E(a_{\kay+m+1})\\
X(z)
&=\sum_{n\geq 0}z^n E(a_1)E(a_2)\cdots E(a_{n+1})\\
Y(z)
&=\sum_{0\leq n<\kay}z^nE(a_1)E(a_2)\cdots E(a_{n+1})
\end{align}
In particular, since 
$X(z)=\left[\begin{array}{cc}F_p(z)&z^{-1}(F_p(z)-p_0)\\F_q(z)&z^{-1}(F_q(z)-q_0)\end{array}\right]$
is a rational function in the variable $z$ with integer coefficients, it follows that $W(z)$ is a rational function 
with integer coefficients.
\end{corollary}
The matrix polynomial $U(z)$ is invertible provided that $z$ is non-zero (and in fact the above expression for $W(z)$ 
obviously can be continued to this case).
The invertibility of $V(z)$ is determined by the following.
\begin{proposition}
\begin{align}
&\det V(z)=\\
&
\det\! \left[\begin{array}{cc}
p_{0}^{\caret\kay} & p_{\ell}^{\caret\kay}\\
q_{0}^{\caret\kay} & q_{\ell}^{\caret\kay}
\end{array}\right] z^{\ell-1}
\!+\!
\sum_{0<m<\ell}
\left(\!
\det\!\left[\begin{array}{cc}
p_{0}^{\caret\kay} & p_{m}^{\caret\kay}\\
q_{0}^{\caret\kay} & q_{m}^{\caret\kay}
\end{array}\right] z^{m-1}
\!+\!
\det\!\left[\begin{array}{cc}
p_{m}^{\caret\kay} & p_{\ell}^{\caret\kay}\\
q_{m}^{\caret\kay} & q_{\ell}^{\caret\kay}
\end{array}\right] z^{\ell+m-1}
\!\right)\notag
\end{align}
\end{proposition}
\begin{proof}
Let 
$V(z)
=
\left[\begin{array}{cc}V_{11}(z)& V_{12}(z)\\ V_{21}(z)& V_{22}(z)\end{array}\right]$.
By the Remark~\ref{rmk:N1_in_terms_of_pq} above we have that 
$E(a_{\kay+1})\cdots E(a_{\kay+m+1})
=\left[\begin{array}{cc}p_{m}^{\caret\kay}&p_{m+1}^{\caret\kay}\\q_{m}^{\caret\kay}&q_{m+1}^{\caret\kay}\end{array}\right]$.
This observation, together with the definition of $V(z)$ gives
\begin{equation}
V(z)
=
\sum_{0\leq m<\ell}z^m 
E(a_{\kay+1})\cdots E(a_{\kay+m+1})
=
\sum_{0\leq m<\ell}z^m 
\left[\begin{array}{cc}p_{m}^{\caret\kay}&p_{m+1}^{\caret\kay}\\q_{m}^{\caret\kay}&q_{m+1}^{\caret\kay}\end{array}\right]
\end{equation}
Therefore we have the following expressions for the entries of $V(z)$:
\begin{align}
V_{11}(z)
&=&\sum_{0\leq m< \ell} p_{m}^{\caret\kay} z^m
&=&p_{0}^{\caret\kay}+P\\
V_{12}(z)
&=&\sum_{0\leq m< \ell} p_{m+1}^{\caret\kay} z^m
&=&z^{-1}(P+p_{\ell}^{\caret\kay} z^\ell)\\
V_{21}(z)
&=&\sum_{0\leq m< \ell} q_{m}^{\caret\kay} z^m
&=&q_{0}^{\caret\kay}+Q\\
V_{22}(z)
&=&\sum_{0\leq m< \ell} q_{m+1}^{\caret\kay} z^m
&=&z^{-1}(Q+q_{\ell}^{\caret\kay} z^\ell)
\end{align}
where
\begin{equation}
P=\sum_{0<m< \ell} p_{m}^{\caret\kay} z^m \qquad
Q=\sum_{0<m< \ell} q_{m}^{\caret\kay} z^m
\end{equation}
Consequently,
\begin{align}
&\det V(z)
=z^{-1}\left(p_{0}^{\caret\kay}+P\right)\left(Q+q_{\ell}^{\caret\kay} z^\ell\right)
-z^{-1}\left(q_{0}^{\caret\kay}+Q\right)\left(P+p_{\ell}^{\caret\kay} z^\ell\right)\\
&=z^{-1}\left(\left(p_{0}^{\caret\kay} Q-q_{0}^{\caret\kay}P\right)
+z^\ell\left[\left(Pq_{\ell}^{\caret\kay} -Qp_{\ell}^{\caret\kay}\right)
+\left(p_{0}^{\caret\kay} q_{\ell}^{\caret\kay}-q_{0}^{\caret\kay} p_{\ell}^{\caret\kay}\right)\right]\right)
\end{align}
This can be expressed as the sum of three determinants of $2\times 2$ matrices.
Multiplying out the first two (using multilinearity of the determinant) 
and collecting terms of the same degree gives the result.
\end{proof}
\begin{remark}\label{rmk:Fix_in_terms_of_det/tr-d=2}
Consider the matrix $N_1=E(a_{\kay+1})\cdots E(a_{\kay+\ell})$ given above.
Since $\det(N_1)=(-1)^\ell$ we find that
\begin{equation}
\det(\mathrm{id}-N_1^n)
=
1-\trace(N_1^n)+\det(N_1^n)
=
1+(-1)^{n\ell}-\trace(N_1^n)
\end{equation}
Recall that $N_1$ is hyperbolic. 
Thus $N_1^n$ is also hyperbolic whenever $n$ is a positive integer. 
Therefore, by considering the discriminant of the characteristic polynomial of $N_1^n$ or otherwise, 
$|\trace(N_1^n)|>2$ for all positive integers $n$.
Moreover, from Remark~\ref{rmk:N1_in_terms_of_pq} above it follows that $\trace(N_1^n)$ is positive.
Hence
\begin{equation}
\left|\det (\mathrm{id}-N_1^n)\right|
=
\trace(N_1^n)-\left(1+(-1)^{n\ell}\right)
\end{equation}
\end{remark}
%
By combining the results above, we are now in a position to prove Theorem~\ref{thm:main_thm}.
\begin{proof}[Proof of Theorem~\ref{thm:main_thm}]
By Corollary~\ref{cor:W-rational} above
\begin{equation}
W(z)=U(z)^{-1}(X(z)-Y(z))V(z)^{-1}
\end{equation}
where $U(z)$, $V(z)$, $X(z)$ and $Y(z)$ are the rational matrix functions as defined in Corollary~\ref{cor:W-rational}.
We take the trace of both sides of this equality. 
One of the expressions is immediate.
Applying the definition of $W(z)$ given in Proposition~\ref{prop:mx_gen_fn-periodic},
followed by Remark~\ref{rmk:Fix_in_terms_of_det/tr-d=2}
and equation~\eqref{eq:Fixfn=|det|}, 
gives us
\begin{align}
&
\trace W(z)
=
\sum_{n\geq 0} \trace (N_{1}^{n})z^{n\ell}\\
&=
\trace (N_{1}^{0})+z^\ell \sum_{n\geq 1}\left|\det(\mathrm{id}-N_{1}^{n})\right|z^{n\ell-\ell}+\sum_{n\geq 1} (1+(-1)^{n\ell})z^{n\ell}\\
&=
\trace (N_{1}^{0})+z^\ell \sum_{n\geq 1}\card\mathrm{Fix}(f^n)z^{n\ell-\ell}+\sum_{n\geq 1} z^{n\ell} +\sum_{n\geq 1}(-z)^{n\ell}\\
&=
2+z^\ell(\log \zeta_{f})'(z^\ell)+\frac{z^\ell}{1-z^\ell}+\frac{(-z)^{\ell}}{1-(-z)^{\ell}}
\end{align}
Combining these equalities and rearranging, we arrive at equation~\eqref{eq:dlogzeta_vs_F-G}, as required. 
Thus the theorem is proved.
\end{proof}
%
%
%
%
%
%
%
%

\appendix
\section{The Gauss transformation and the theorem of L{\'e}vy.}\label{sect:Gauss_Map}
Let $\mathrm{T}$ denote the Gauss transformation on the interval $[0,1]$, {\it i.e.\/}, the transformation
\begin{equation}
\mathrm{T}(\theta)=
\left\{\begin{array}{ll}
\left\{\frac{1}{\theta}\right\} & \theta\in (0,1]\\
0 & \theta=0
\end{array}\right.
\end{equation}
where $\{x\}$ denotes the fractional part of the real number $x$.
The Gauss transformation possesses an ergodic absolutely continuous invariant probability measure $\mu$ given explicitly by
\begin{equation}
\mu=\frac{1}{\log 2}\frac{dx}{1+x}
\end{equation}
For $\theta\in[0,1]\setminus\mathbb{Q}$ and any positive integer $n$,
\begin{equation}
\theta=[a_1,a_2,\ldots,a_n+\mathrm{T}^n(\theta)]
\end{equation}
Therefore, 
\begin{equation}
\theta=\frac{p_{n-1}\mathrm{T}^n(\theta)+p_n}{q_{n-1}\mathrm{T}^n(\theta)+q_n}
\end{equation}
Inverting gives
\begin{equation}
\mathrm{T}^n(\theta)=-\frac{q_n\theta-p_n}{q_{n-1}\theta-p_{n-1}}=\frac{|\Delta_n|}{|\Delta_{n-1}|}
\end{equation}
where $\Delta_n$ denotes the $n$th short renormalisation interval.
Thus
\begin{equation}\label{eq:short_interval_vs_Gamma}
\frac{|\Delta_n|}{|\Delta_0|}=\prod_{m=1}^n \mathrm{T}^m(\theta) 
\end{equation}
and so by Birkhoff's Ergodic Theorem, for $\mu$-almost every $\theta\in  [0,1]$
\begin{equation}
\frac{1}{n}\log |\Delta_n|
=\frac{1}{n}\left[\log |\Delta_0|+\sum_{m=1}^n \log \mathrm{T}^m(\theta)\right]
\longrightarrow \int_{[0,1]} \log x \ d\mu
=\frac{\pi^2}{12\log 2}
\end{equation}
as $n$ tends to infinity. 
Observe that, since $\mu$ is absolutely continuous with respect to Lebesgue, the null sets for $\mu$ are exactly the null sets for Lebesgue measure.
Hence the above limit holds also for Lebesgue-almost every $\theta\in[0,1]$.
Since 
\begin{equation}
\frac{1}{q_nq_{n+2}}\leq\frac{1}{q_n(q_n+q_{n+1})}<\left|\theta-\frac{p_n}{q_n}\right|\leq \frac{1}{q_nq_{n+1}}
\end{equation}
we find that
\begin{equation}\label{eq:q_n_vs_Delta_n-sandwich}
\frac{1}{n}\log|\Delta_n|
=\frac{1}{n}\log |q_n\theta-p_n|
\leq
-\frac{1}{n}\log q_{n+1}
\leq
\frac{1}{n}\log|\Delta_{n-1}|
\end{equation}
Exponentiating, we then find that
\begin{equation}
\lim_{n\to\infty} q_n^{1/n}=e^{-\pi^2/12\log 2} \qquad \mbox{for Lebesgue-almost all} \  \theta\in [0,1]
\end{equation}
Since $|q_n\theta-p_n|\leq 1$ we also find that 
\begin{equation}
\lim_{n\to\infty} p_n^{1/n}=e^{-\pi^2/12\log 2} \qquad \mbox{for Lebesgue-almost all} \ \theta\in[0,1]
\end{equation}
We call $e^{-\pi^2/12\log 2}$ the {\it L{\'e}vy constant}.
See~\cite[p.66]{KhinchinBook} and~\cite[p.320]{LevyBook}.

Now consider the special case 
of quadratic irrationals.
Take a quadratic irrational $\theta$ as above, 
so there exists 
a positive integer $\ell$ and 
a non-negative integer $\kay$ such that 
$\mathrm{T}^\kay(\theta)=\mathrm{T}^{\kay+\ell}(\theta)$ and let $\theta^{\caret \kay}=\mathrm{T}^\kay(\theta)$.
Then equation~\eqref{eq:short_interval_vs_Gamma} implies that, for any positive $j<\ell$ and non-negative $n$,
\begin{align}
\log |\Delta_{\kay+j+n\ell}|
&=\log |\Delta_0|+\sum_{r=0}^{\kay+j+n\ell}\log\mathrm{T}^r(\theta)\\
&=O(1)+\sum_{r=0}^n\sum_{s=0}^{\ell-1}\log\mathrm{T}^{s+r\ell}(\theta^{\caret \kay})
\end{align}
From which we find, applying inequality~\eqref{eq:q_n_vs_Delta_n-sandwich} that
\begin{equation}
\lim_{n\to\infty}\frac{1}{n}\log q_n=-\frac{1}{\ell}\sum_{s=0}^{\ell-1}\log\mathrm{T}^s(\theta^{\caret \kay})
\end{equation}
Thus, in particular, the L{\'e}vy constant $\beta(\theta)$ is defined.
This statement was first proved by Jager and Liardet~\cite{JagerLiardet1988}.


\begin{thebibliography}{999}
\bibliographystyle{plain}



\bibitem{BelovaHazard2017a}
\newblock A.~Belova and P.~Hazard.
\newblock {\it Quadratic Irrationals, Generating Functions and L{\'e}vy Constants\/}.
\newblock ArXiv Preprint \texttt{arXiv:1710.08990v2}, 2018.

\bibitem{Bowen1970b}
\newblock R.~Bowen.
\newblock {\it Topological entropy and Axiom {A}\/}.
\newblock in {\it Global Analysis}, Proc. Symp. Pure. Math., Amer. Math. Soc., \textbf{14}, 23--42.

\bibitem{Bowen1971} 
\newblock R.~Bowen.
\newblock {\it Entropy for Group Endomorphisms and Homogeneous Space\/}.
\newblock Trans. Amer. Math. Soc., \textbf{153}, (1971), 401--414. 

\bibitem{Bowen1978b} 
\newblock R.~Bowen.
\newblock {\it On Axiom A diffeomorphisms (Conf. Board Math. Sciences)\/}.
\newblock vol. \textbf{35}, Amer. Math. Soc., Providence, RI, 1978. 

\bibitem{Faivre1992}
\newblock C.~Faivre.
\newblock {\it Distribution of L{\'e}vy constants for quadratic numbers\/}.
\newblock Acta Arith. LXI.1, 1992, 13--34.


\bibitem{FarbMargalitBook2012}
\newblock B.~Farb and D.~Margalit.
\newblock {\it A primer on mapping class groups\/}.
\newblock Princeton University Press, 2012.



\bibitem{GunningBook1962}
\newblock R. C.~Gunning.
\newblock {\it Lectures on Modular Forms\/}.
\newblock Annals of Mathematics Studies, \textbf{AM--48}, Princeton University Press, 1962.

\bibitem{HardyWrightBook}
\newblock G. H.~Hardy and E. M.~Wright.
\newblock {\it An introduction to the theory of number (6th ed.)\/}.
\newblock Oxford University Press, 2008.

\bibitem{JagerLiardet1988}
\newblock H.~Jager and P.~Liardet.
\newblock {\it Distributions arithm{\'e}tiques des d{\'e}nominateurs de convergents de fractions continues\/}.
\newblock Indag. Math., \textbf{50}, (1988), 181--197.
 
\bibitem{KhinchinBook}
\newblock A. Ya.~Khinchin.
\newblock {\it Continued Fractions\/}, (transl. H.~Eagle).
\newblock Dover Publications, 1997.


\bibitem{Lehmer1939}
\newblock D. H.~Lehmer.
\newblock {\it Note on an absolute constant of Khintchine\/}.
\newblock Amer. Math. Monthly, vol. 46, no. 3, (1939), 148--152.

\bibitem{Levy1929}
\newblock P.~L{\'e}vy.
\newblock {\it Sur les lois de probabilit{\'e} dont d{\'e}pendent les quotients complets et incomplets d'un fraction continue\/}.
\newblock Bull. Soc. Math. de France, tome \textbf{57} (1929), 178--194.

\bibitem{LevyBook}
\newblock P.~L{\'e}vy.
\newblock {\it Th{\'e}orie de l'addition des variables al{\'e}atoire\/}.
\newblock Gauthier-Villars Paris, 1937.

\bibitem{LindMarcusBook}
\newblock D.~Lind and B.~Marcus.
\newblock {\it An introduction to symbolic dynamics and coding\/}.
\newblock Cambridge University Press, 1999. 

\bibitem{ParryPollicott1990}
\newblock W.~Parry and M.~Pollicott.
\newblock {\it Zeta functions and the periodic orbit structure of hyperbolic dynamics\/}.
\newblock Ast\'erisque, \textbf{187--188}, 1990.

\bibitem{Pollicott1986}
\newblock M.~Pollicott.
\newblock {\it Distribution of closed geodesics on the modular surface and quadratic irrationals\/}.
\newblock Bull. Soc. math. France, 114, (1986), 431--446.


\bibitem{Sarnak1982}
\newblock P.~Sarnak.
\newblock {\it Class numbers of indefinite binary quadratic forms\/}.
\newblock J. Number Theory, vol. 15, (1982), 229--247.

\bibitem{Sarnak2007}
\newblock P.~Sarnak.
\newblock {\it Reciprocal geodesics\/}.
\newblock in {\it Analytic Number Theory}, Clay Math. Proc., vol. 7, Amer. Math. Soc., Providence RI, (2007), 217--237.

\bibitem{Series1984}
\newblock C.~Series.
\newblock {\it The modular surface and continued fractions\/}.
\newblock J. London Math. Soc., (2), 31, (1985), 69--80.

\bibitem{WaltersBook}
\newblock P.~Walters.
\newblock {\it An introduction to ergodic theory\/}.
\newblock Graduate Texts in Mathematics \textbf{79}, Springer-Verlag, 1981.



\end{thebibliography}
\end{document}